\documentclass[11pt,reqno,a4paper]{amsart}

\usepackage{docmute}

\usepackage{graphicx}
\usepackage{midpage}
\usepackage[margin=25truemm]{geometry}
\usepackage{stix2}
\usepackage{amsmath,amssymb,amsthm,amsfonts}
\usepackage{tikz-cd}
\usepackage{setspace}
\makeatletter
\DeclareRobustCommand\showhyphens[1]{%
	\setbox0\vbox{%
		\color@begingroup
		\everypar{}%
		\parfillskip\z@skip\hsize\maxdimen
		\normalfont
		\pretolerance\m@ne\tolerance\m@ne\hbadness\z@\showboxdepth\z@\ #1%
		\color@endgroup}}
\makeatother
\usepackage[protrusion=true,expansion=true,final]{microtype}
\usepackage{hyperref}
\usepackage{lscape}
\usepackage{enumitem}
\usepackage{caption}
\usepackage{array}
\usepackage{makecell}
\usepackage{environ}
\usepackage{subfiles}

\numberwithin{equation}{subsection}
\usetikzlibrary{calc}

\hypersetup{% hyperrefオプションリスチE
setpagesize=false,
 bookmarksnumbered=true,%
 bookmarksopen=true,%
 colorlinks=true,%
 linkcolor=red,
 citecolor=blue,
}

\let\it\relax
\let\Im\relax
\newcommand{\it}[1]{\textit{#1}}

\newcommand{\op}{^\mathrm{op}}

\DeclareMathOperator{\id}{\mathrm{id}}

\let\hat\relax

\newcommand{\dgcat}{\mathrm{dgcat}}
\newcommand{\hqe}{\mathrm{Hqe}}
\newcommand{\dg}{_{\mathrm{dg}}}

\newcommand{\ox}{\otimes}
\newcommand{\rhom}[2]{\mathbb{R}\mathrm{Hom}_{#1}(#2)}
\newcommand{\rend}[2]{\mathbb{R}\mathrm{End}_{#1}(#2)}
\newcommand{\lox}{\otimes^{\mathbb{L}}}
\newcommand{\lquot}{/^{\mathbb{L}}}

\newcommand{\hat}{_\land}
\newcommand{\hut}{_\vee}

\newcommand{\calA}{\mathcal{A}}
\newcommand{\calB}{\mathcal{B}}
\newcommand{\calC}{\mathcal{C}}
\newcommand{\calD}{\mathcal{D}}

\newcommand{\calK}{\mathcal{K}}

\newcommand{\calS}{\mathcal{S}}
\newcommand{\calT}{\mathcal{T}}
\newcommand{\calU}{\mathcal{U}}

\newcommand{\bbU}{\mathbb{U}}
\newcommand{\bbV}{\mathbb{V}}
\newcommand{\bbW}{\mathbb{W}}

\newcommand{\bbZ}{\mathbb{Z}}

\newcommand{\scrA}{\mathscr{A}}
\newcommand{\scrB}{\mathscr{B}}
\newcommand{\scrC}{\mathscr{C}}
\newcommand{\scrD}{\mathscr{D}}

\newcommand{\scrK}{\mathscr{K}}

\newcommand{\scrP}{\mathscr{P}}

\newcommand{\scrT}{\mathscr{T}}
\newcommand{\scrU}{\mathscr{U}}

\newcommand{\scrX}{\mathscr{X}}
\newcommand{\scrY}{\mathscr{Y}}
\newcommand{\scrZ}{\mathscr{Z}}

\newcommand{\dem}[1]{\mathcal{H}^d(#1)}
\newcommand{\demfp}[1]{\mathcal{H}^d_{\mathrm{fp}}(#1)}

\newcommand{\demdg}[1]{\mathscr{H}^d(#1)}
\newcommand{\demfpdg}[1]{\mathscr{H}^d_{\mathrm{fp}}(#1)}
\newcommand{\demfddg}[1]{\mathscr{H}^d_{\mathrm{fd}}(#1)}
\newcommand{\der}[1]{\mathcal{D}(#1)}

\newcommand{\Ar}[3]{\ar[from=#1,to=#2,#3]}

\newcommand{\cma}{\mathpunct{,}}
\newcommand{\tauAr}[3]{\ar[from={#1}, to={#2}, #3,color=gray, dash pattern=on 3pt off 4pt]}

\DeclareMathOperator{\pvd}{\mathrm{pvd}}
\DeclareMathOperator{\per}{\mathrm{per}}
\DeclareMathOperator{\Mod}{\mathrm{Mod}}
\DeclareMathOperator{\rep}{\mathrm{rep}}

\DeclareMathOperator{\Kerd}{\mathrm{Ker}_{\it{d}}}
\DeclareMathOperator{\Cok}{\mathrm{Cok}}
\DeclareMathOperator{\Cokd}{\mathrm{Cok}_{\it{d}}}
\DeclareMathOperator{\Im}{\mathrm{Im}}

\DeclareMathOperator{\Cone}{\mathrm{Cone}}
\DeclareMathOperator{\Cocone}{\mathrm{CoCone}}

\DeclareMathOperator{\add}{\mathrm{add}}

\DeclareMathOperator{\End}{\mathrm{End}}

\DeclareMathOperator{\cpx}{\mathrm{Cpx}}

\NewEnviron{xsmallmatrix}{%
	\vcenter{\hbox{\scalebox{1.20}{$\renewcommand\thickspace{\kern.0em}\begin{smallmatrix}\BODY\end{smallmatrix}$}}}%
}

\theoremstyle{definition}

\newtheorem{theorem}{Theorem}[section]
\newtheorem*{theorem*}{Theorem}
\newtheorem{proposition}[theorem]{Proposition}
\newtheorem*{proposition*}{Proposition}

\newtheorem{lemma}[theorem]{Lemma}
\newtheorem{cor}[theorem]{Corollary}
\newtheorem*{cor*}{Corollary}

\newtheorem{definition}[theorem]{Definition}
\newtheorem*{definition*}{Definition}

\newtheorem{remark}[theorem]{Remark}
\newtheorem{example}[theorem]{Example}

\newtheorem{fact}[theorem]{Fact}

\begin{document}

\title[Extended Module Categories in Higher Cluster Tilting Theory]{Extended Module Categories in Higher Cluster Tilting Theory}

% \author{Nao Mochizuki}

% \address{Graduate School of Mathematics, Nagoya University, Furocho, Chikusaku, Nagoya 464-8602, Japan}
% \email{mochizuki.nao.n8@s.mail.nagoya-u.ac.jp}

\author[Nao Mochizuki]{Nao Mochizuki}
	\address{Graduate School of Mathematics, Nagoya University, Furocho, Chikusa-ku, Nagoya 464-8602, Japan}
	\email{mochizuki.nao.n8@s.mail.nagoya-u.ac.jp} %

\begin{abstract}
In this paper, we study ideal quotients of triangulated categories by higher cluster tilting subcategories. Koenig and Zhu proved that the ideal quotient by a $2$-cluster tilting subcategory is an abelian category; moreover, by Morita's theorem, it is equivalent to the module category over the $2$-cluster tilting subcategory. We generalize this result to higher cluster tilting subcategories. More precisely, we show that the natural DG-enhancement of the ideal quotient of a triangulated category by a $(d+1)$-cluster tilting subcategory is an abelian $d$-truncated DG-category. In the appendix, we prove a Morita-type theorem for abelian $d$-truncated DG-categories, which asserts that an abelian $d$-truncated DG-category with enough projectives is equivalent to a $d$-extended module category over a $d$-truncated DG-category. As an application, we show that the ideal quotient of a triangulated category by a $(d+1)$-cluster tilting subcategory is equivalent to a $d$-extended module category over a
$d$-truncated DG-category.
\end{abstract}

\maketitle

\tableofcontents

\setcounter{section}{-1}

\section{Introduction}

Higher cluster tilting subcategories in triangulated categories arise naturally in higher Auslander--Reiten theory and higher cluster category theory \cite{IY2008,GKO2013}. 
These subcategories have rich homological structures. In particular, the theory of $2$-cluster tilting subcategories is well developed. A fundamental result due to Koenig and Zhu \cite{KZ2008} states that the ideal quotient of a triangulated category by a $2$-cluster tilting subcategory is an abelian category with enough projectives. This result plays an important role in the study of $2$-cluster tilting subcategories in cluster categories \cite{BMR2007,KR2007}. 

A natural question is what replaces this ideal quotient for higher cluster tilting subcategories. For $d>1$, one should not expect the ordinary ideal quotient to be an abelian category. Instead, the higher homological information that is invisible in the ordinary quotient is naturally encoded in the negative cohomology of a DG-enhancement of the quotient. This leads to the notion of an abelian $d$-truncated DG-category, introduced by the author \cite{moc2025}, as a DG-categorical analogue of the abelian $(d,1)$-category introduced by Stefanich \cite{ste2023}.

In this paper, we prove that, for a $(d+1)$-cluster tilting subcategory of an algebraic triangulated category, the DG-enhanced ideal quotient is an abelian $d$-truncated DG-category. To formulate and prove this result, we work mainly with DG-categories.

More precisely, we use stable DG-categories, introduced by Chen \cite{che2023}, rather than the usual pretriangulated DG-categories. Stable DG-categories are DG-categorical analogues of stable $(\infty,1)$-categories. The point of using them here is that their connective DG-quotients enhance ideal quotients of homotopy categories, whereas pretriangulated DG-quotients naturally enhance Verdier quotients. This distinction is essential for our main theorem.

Our main result is the following:
\begin{theorem}[$=$ Theorem~\ref{thm:preabelian} $+$ Proposition~\ref{prop:enough_proj} $+$ Proposition~\ref{prop:enough_inj}]
Let $\scrT$ be an essentially small, locally $K$-projective and stable DG-category. Let $\scrC$ be a full DG-subcategory such that $H^0(\scrC)$ is a $(d+1)$-cluster tilting subcategory of $H^0(\scrT)$. Then the DG-quotient $\scrT\lquot\scrC$ is an abelian $d$-truncated DG-category with enough projectives and injectives. Moreover, an object in $\scrT\lquot\scrC$ is projective (resp. injective) if and only if it is isomorphic to a direct summand of $Q(X[-d])$ (resp. $Q(X[d])$) for some $X\in\scrC$, where $Q\colon \scrT\to \scrT\lquot\scrC$ is the DG-quotient morphism in $\hqe$.
\end{theorem}

We also prove the corresponding Morita theorem. In Koenig--Zhu's theorem, the
abelian quotient is described as a module category over the cluster tilting
subcategory. Our Morita theorem shows that the same philosophy persists in the
higher setting: abelian $d$-truncated DG-categories with enough projectives
are precisely finitely presented $d$-extended module categories. The following theorem was already known, in the finite-length case, by \cite{plogmann2026}.

% 以下の定理は，長さ有限のアーベルd-truncated DG-圏の場合はすでに知られている．

\begin{theorem}[$=$ Corollary~\ref{cor:morita}]
Let $\scrA$ be a locally $K$-projective abelian $d$-truncated DG-category with an essentially small homotopy category. Let $\scrB$ be a progenerating full DG-subcategory in $\scrA$. Then the following DG-functor is a quasi-equivalence:
$$\scrA(-,*)|_{\scrB} \colon \scrA \xrightarrow{\sim} \demfpdg{\scrB}$$
where $\demfpdg{\scrB}$ is the finitely presented $d$-extended module category (Definition~\ref{def:fp}) over $\scrB$.
\end{theorem}

Combining these two theorems, we obtain the following higher analogue of the module-theoretic description in Koenig--Zhu's theorem.

\begin{cor}
Let $\calT$ be an essentially small algebraic triangulated category and $M\in \calT$ a $(d+1)$-cluster tilting object. Let $\scrT$ be a stable DG-category such that $H^0(\scrT)\simeq \calT$, $\Lambda:=\tau^{> -d}\End_{\scrT}(M)$ the DG-endomorphism algebra of $M$. Then there exists an equivalence of categories:
$$\calT/[\add M] \simeq \demfp{\Lambda}$$
where $\calT/[\add M]$ is the ideal quotient of $\calT$ by $M$ and $\demfp{\Lambda}$ is the finitely presented $d$-extended module category over $\Lambda$. 
\end{cor}

We also apply the general result to Amiot--Guo's generalized cluster
categories \cite{Amiot2009,Guo11}. In this case, the quotient can be described explicitly in terms of
a truncation of the Calabi--Yau DG-algebras. 

\subsection*{Acknowledgements}
The author is grateful to Prof. Yasuaki Ogawa for suggesting this problem and for many helpful discussions. The author also thanks Prof. Hiroyuki Nakaoka, Prof. Osamu Iyama, and Junyang Liu for valuable comments. This work was financially supported by JST SPRING, Grant Number JPMJSP2125.

% The notion of abelian $d$-truncated DG-category is introduced by the author and is a DG-categorical analogue of the abelian $d$-category introduced by Stefanich \cite{ste2023}. When $d=1$, it coincides with the usual abelian category viewed as a DG-category. Therefore, our result is a generalization of the Koenig--Zhu result to the case of higher cluster tilting subcategories.

\section{Preliminaries}

\subsection{Size conventions of DG-categories}
In this subsection, we fix a universe $\bbU$ satisfying the axiom of infinity. We recall some results on the homotopy theory of DG-categories. For details, see \cite{gon2005,toe2007}.

\begin{fact}[{\cite{gon2005}}]
\label{thm:model_str_dgcat}
There exists a cofibrantly generated model structure on $\dgcat_{\bbU}$ such that the weak equivalences are the quasi-equivalences. We denote by $\hqe_{\bbU}$ the homotopy category of $\dgcat_{\bbU}$ with respect to this model structure. 
\end{fact}

\begin{remark}
We call a $\bbU$-small DG-category \emph{$\bbU$-cofibrant} if it is cofibrant with respect to the above model structure. If $\bbU\in \bbV$ are universes satisfying the axiom of infinity, then any $\bbU$-cofibrant DG-category is $\bbV$-cofibrant because the generating cofibrations and the generating trivial cofibrations in $\dgcat_{\bbU}$ are also generating cofibrations and generating trivial cofibrations in $\dgcat_{\bbV}$. For details, see \cite{gon2005,toe2007}. This implies that the homotopy category $\hqe_{\bbU}$ is a full subcategory of $\hqe_{\bbV}$ and the embedding $\hqe_{\bbU}\subset \hqe_{\bbV}$.
\end{remark}

\begin{remark}
Since the above model structure is cofibrantly generated, we have a functorial cofibrant replacement functor $Q\colon \dgcat_{\bbU}\to \dgcat_{\bbU}$ such that $Q(\scrA)$ is $\bbU$-cofibrant for each $\bbU$-small DG-category $\scrA$. Moreover, we also have a natural quasi-equivalence $Q(\scrA)\to \scrA$ for each $\bbU$-small DG-category $\scrA$, which is the identity on objects. By using this functorial cofibrant replacement functor, we can construct the monoidal structure $-\lox-\colon \hqe_{\bbU}\times \hqe_{\bbU}\to \hqe_{\bbU}$ on $\hqe_{\bbU}$ by putting $\scrA\lox\scrB:=Q(\scrA)\ox \scrB$. A morphism $f\colon X\to Y$ in a DG-category $\scrA$ is called \emph{closed} if $f$ has degree $0$ and $d(f)=0$. A closed morphism $f\colon X\to Y$ in $\scrA$ is called an \emph{isomorphism} if it is an isomorphism in the homotopy category $H^0(\scrA)$. 
\end{remark}

\begin{definition}
\label{def:dg_module}
Let $\scrA$ be a $\bbU$-small DG-category. A \emph{right DG-module} over $\scrA$ is a DG-functor $M\colon \scrA\op\to \scrC_\bbU(k)$, where $\scrC_\bbU(k)$ is the DG-category of $\bbU$-small complexes of $k$-modules. If a right DG-module $M$ satisfies that $M(A)$ is an acyclic complex for each $A\in\scrA$, then we call $M$ an \emph{acyclic right DG-module}. A right DG-module $M$ is called \emph{$K_{\bbU}$-projective} if Hom-complexes from $M$ to any acyclic right DG-module are acyclic. 
\end{definition}

\begin{definition}
\label{def:derived_cat}
Let $\scrA$ be a $\bbU$-small DG-category. We define the \emph{derived DG-category} $\scrD_\bbU (\scrA)$ of $\scrA$ as the $\bbU$-small DG-category of $K_{\bbU}$-projective right DG-modules over $\scrA$. The homotopy category $\calD_\bbU (\scrA)$ of $\scrD_\bbU (\scrA)$ is called the \emph{derived category} of $\scrA$.
\end{definition}

\begin{remark}
Let $\bbU\in \bbV$ be universes satisfying the axiom of infinity. Assume that $\scrA$ is a $\bbU$-small DG-category and $M\colon \scrA\op\to \scrC_\bbU(k)$ is a right DG-module over $\scrA$. Then $M$ is a $K_{\bbU}$-projective right DG-module over $\scrA$ if and only if $M$ is a $K_{\bbV}$-projective right DG-module over $\scrA$. This is because the closed morphisms from $M$ to any acyclic right DG-module $N\colon \scrA\op\to \scrC_\bbV(k)$ factor through a submodule $N'\colon \scrA\op\to \scrC_\bbV(k)$ of $N$ such that $N'(A)$ is a $\bbV$-small acyclic complex for each $A\in\scrA$. Thus, we can regard the derived DG-category $\scrD_\bbU (\scrA)$ of $\scrA$ as a full DG-subcategory of $\scrD_\bbV (\scrA)$. In this sense, $K$-projectivity of a right DG-module over $\scrA$ does not depend on the choice of universes. Thus, we can simply say that $M$ is a $K$-projective right DG-module over $\scrA$ without specifying the universe.
\end{remark}

\begin{remark}
Let $\bbU\in \bbV$ be universes satisfying the axiom of infinity. If $\scrA$ is a $\bbU$-small DG-category, then the derived DG-category $\scrD_\bbU (\scrA)$ is $\bbV$-small. 
\end{remark}

\begin{fact}[{\cite[Theorem~6.1]{toe2007}}]
\label{thm:internal_hom}
The category $\hqe_{\bbU}$ admits an internal hom functor $\rep_{\bbU}(-,-)$ which is defined as follows: 
$$\rep_{\bbU}(\scrA,\scrB):=\{M\in \der{\scrA\op\lox\scrB}\mid M \text{ is right quasi-representable}\}.$$
Here, a DG-bimodule $M\in \der{\scrA\op\lox\scrB}$ is called \emph{right quasi-representable} if for each $A\in\scrA$, there exists $B\in\scrB$ such that $M(A,-)\cong \scrB(-,B)$ in $\der{\scrB}$. 

Moreover, the isomorphism classes in the homotopy category of $\rep_{\bbU}(\scrA,\scrB)$ are in bijection with the morphisms in $\hqe_{\bbU}$ from $\scrA$ to $\scrB$.
\end{fact}

The locally $K$-projective condition is more tractable than the cofibrancy condition. 

\begin{definition}
\label{def:locally_K_projective}
A $\bbU$-small DG-category $\scrA$ is called \emph{locally $K$-projective} if for each pair of objects $A,B\in\scrA$, the Hom-complex $\scrA(A,B)$ is a $K$-projective complex of $k$-modules.
\end{definition}

\subsection{DG-quotients of DG-categories}

In this subsection, we recall the notion of DG-quotients of DG-categories and their properties. For details, see \cite{Kel1999,dri2004}. We continue to work with a universe $\bbU$ satisfying the axiom of infinity. Let $\scrT$ be a locally $K$-projective $\bbU$-small DG-category and $\scrC$ a full DG-subcategory of $\scrT$.

% \color{blue}
% \begin{notation}
% We note the following notations in this paper:
% \begin{enumerate}
% 	\item We denote by $\calT$ and $\calC$ the homotopy categories $H^0(\scrT)$ and $H^0(\scrC)$ respectively.
% 	\item We denote $\calD(\scrT)$ the derived category of $\scrT$. 
% 	\item We denote a representable DG-module by $X\hat:=\scrT(-,X)\in \calD(\scrT)$ for $X\in\scrT$.
% 	\item Let $F\colon \scrT\to \scrT'$ be a DG-functor and $M\in \calD(\scrT')$ a right DG-module over $\scrT'$. We denote by $M|_{\scrT}$ the restriction of $M$ along $F$, i.e., $M|_{\scrT}:=M\circ F\in \calD(\scrT)$.
% \end{enumerate}
% \end{notation}
% \normalcolor

\begin{fact}(\cite[Proposition~4.6]{dri2004})
\label{thm:recollement_der_quot}
There exists a $\bbU$-small DG-category $\scrT\lquot\scrC$ and a DG-functor $\scrT\to \scrT\lquot\scrC$ which induces a recollement of triangulated categories:
$$
\begin{tikzcd}[column sep=5em]
\calD(\scrT\lquot\scrC) & \calD(\scrT) & \calD(\scrC)
\Ar{1-1}{1-2}{"(-)|_{\calT}"{description}}
\Ar{1-2}{1-3}{"(-)|_{\calC}"{description}}
\Ar{1-2}{1-1}{"-\lox_{\scrT}(\scrT\lquot\scrC)|_{\scrT\op}", bend left=30}
\Ar{1-3}{1-2}{"-\lox_{\scrC}\scrT|_{\scrC\op}", bend left=30}
\Ar{1-3}{1-2}{"\rhom{\scrC}{\scrT|_{\scrC},-}"', bend right=30}
\Ar{1-2}{1-1}{"\rhom{\scrT}{(\scrT\lquot\scrC)|_{\scrT},-}"', bend right=30}
\end{tikzcd}
$$
We call such a DG-category $\scrT\lquot\scrC$ a \emph{DG-quotient} of $\scrT$ by $\scrC$. We also denote the DG-functor by $Q\colon \scrT\to \scrT\lquot\scrC$ and call it a \emph{quotient DG-functor}.
\end{fact}

\begin{remark}
The DG-quotient $\scrT\lquot\scrC$ is unique up to isomorphism in $\hqe_{\bbU}$ since DG-quotients are defined by a universal property in $\hqe$ (see \cite{tab2010} for details). 
\end{remark}

The following is a characterization of the DG-quotient functor $Q\colon \scrT\to \scrT\lquot\scrC$.

\begin{fact}(\cite[Proposition~1.4]{dri2004})
\label{prop:char_quot_functor}
If a functor $Q\colon \scrT\to \scrU$ satisfies the following conditions, then $Q$ is a quotient DG-functor of $\scrT$ by $\scrC$:
\begin{enumerate}
	\item $Q$ induces an essentially surjective functor $\calT\to \calU$, and any object in $\calC$ is sent to the zero object in $\calU$.
	\item The cocone $\Cocone Q_T$ of the natural morphism $Q_T\colon \scrT(*,T) \to \scrU(Q*,QT)$ satisfies that the counit morphism 
	$$\epsilon_{\Cocone Q_T}\colon (\Cocone Q_T)|_{\scrC}\lox_{\scrC}\scrT|_{\scrC\op} \to \Cocone Q_T$$
	is an isomorphism in $\calD(\scrT)$ for each $T\in\scrT$.  
	\item The dual of (2).
\end{enumerate}
\end{fact}

\begin{cor}
There exists a triangle for each $X\in\calD(\scrT)$:
$$X|_{\scrC}\lox_{\scrC}(\scrT|_{\calC\op}) \xrightarrow{\epsilon_X} X\xrightarrow{\eta_{X}} \bigl(X\lox_{\scrT}(\scrT\lquot\scrC)|_{\scrT\op}\bigr)|_{\scrT}\dashrightarrow$$
where $\epsilon_X$ is a counit morphism of the adjunction $-\lox_{\scrC}\scrT\dashv (-)|_{\scrC}$ and $\eta_X$ is a unit morphism of the adjunction $-\lox_{\scrT}(\scrT\lquot\scrC)\dashv (-)|_{\scrT}$. In particular, if $X\cong T\hut$ for some $T\in\scrT$, then we have the following triangle:
$$
\scrT(-,T)|_{\scrC}\lox_{\scrC}\scrT(*,-)|_{\scrC\op} \xrightarrow{\epsilon_{T\hat}} \scrT(*,T)\xrightarrow{\eta_{T\hat}} (\scrT\lquot\scrC)(Q*,QT)\dashrightarrow
$$
In this case, $\eta_{T\hat}$ corresponds to the natural morphism $Q_T\colon \scrT(*,T)\to (\scrT\lquot\scrC)(Q*,QT)$ induced by the DG-quotient functor $Q\colon \scrT\to \scrT\lquot\scrC$ and $\epsilon_{T\hat}$ corresponds to the natural morphism 
$$m^{\scrC}_T\colon \scrT(-,T)|_{\scrC}\lox_{\scrC}\scrT(*,-)|_{\scrC\op}\to \scrT(*,T)$$ induced by the composition in $\scrT$. Consequently, we have the following triangle in $\calD(\scrT)$:
\begin{equation}\label{eq:standard_tri}
\scrT(-,T)|_{\scrC}\lox_{\scrC}\scrT(*,-)|_{\scrC\op} \xrightarrow{m^{\scrC}_T} \scrT(*,T)\xrightarrow{Q_T} (\scrT\lquot\scrC)(Q*,QT)\dashrightarrow
\end{equation}
\end{cor}

DG-quotients of pretriangulated DG-categories are also pretriangulated, and they naturally enhance Verdier quotients of triangulated categories as follows.

\begin{fact}\cite[Theorem~3.4]{dri2004}
\label{thm:dg-quot_Verdier}
If $\scrT$ is pretriangulated, then the DG-quotient $\scrT\lquot\scrC$ is also pretriangulated, and the induced functor $H^0(\scrT)\to H^0(\scrT\lquot\scrC)$ is a Verdier quotient functor with respect to the triangulated subcategory generated by $\calC$ in $\calT$. In particular, if $\calC$ is a triangulated subcategory of $\calT$, then there exists a triangle equivalence 
$$H^0(\scrT\lquot\scrC)\simeq \calT/\calC.$$
\end{fact}

On the other hand, if $\scrT$ is a connective DG-category, then the DG-quotient $\scrT\lquot\scrC$ is also connective, and it naturally enhances the quotient of the homotopy category $H^0(\scrT)$ by the ideal of morphisms factoring through $\calC$ as follows.

\begin{definition}
We call $\scrT$ \emph{connective} if $H^i(\scrT(X,Y))=0$ for any $X,Y\in\scrT$ and $i>0$.
\end{definition}

\begin{proposition}
\label{prop:derived_ideal_quot}
Let $\scrT$ be a connective DG-category. Then the DG-quotient $\scrT\lquot\scrC$ is also connective and the induced functor $H^0(\scrT)\to H^0(\scrT\lquot\scrC)$ is a quotient functor with respect to the ideal of morphisms factoring through $\calC$ in $\calT$. In particular, there exists a natural equivalence of categories:
$$H^0(\scrT\lquot\scrC)\simeq \calT/[\calC]$$
where $\calT/[\calC]$ is the quotient of $\calT$ by the ideal of morphisms factoring through $\calC$.
\end{proposition}
\begin{proof}
Since $\scrT$ is connective, we have $H^i(\scrT(-,T)|_{\scrC}\lox_{\scrC}\scrT(*,-)|_{\scrC\op})=0$ for $i>0$. Hence the triangle \eqref{eq:standard_tri} implies the following long exact sequence:
$$H^0(\scrT(-,T)|_{\scrC}\lox_{\scrC}\scrT(*,-)|_{\scrC\op}) \to H^0(\scrT(*,T)) \to H^0((\scrT\lquot\scrC)(Q*,QT)|_{\scrT}) \to 0$$
in $\Mod \calT$. Again, since $\scrT$ is connective, we have the following isomorphisms:
$$H^0(\scrT(-,T)|_{\scrC}\lox_{\scrC}\scrT(*,-)|_{\scrC\op}) \cong \calT(-,T)|_{\calC}\ox_{\calC}\calT(*,-)|_{\calC\op}$$
and this isomorphism is compatible with the natural morphism 
$$m^{\calC}_T\colon \calT(-,T)|_{\calC}\ox_{\calC}\calT(*,-)|_{\calC\op}\to \calT(-,T)$$
induced by the composition in $\calT$. Thus, we have the following commutative diagram with exact rows:
$$\begin{tikzcd}
\calT(-,T)|_{\calC}\ox_{\calC}\calT(*,-)|_{\calC\op} & \calT(*,T) & (\calT/[\calC])(*,T) & 0\\
H^0\bigl(\scrT(-,T)|_{\scrC}\lox_{\scrC}\scrT(*,-)|_{\scrC\op}\bigr)  & H^0(\scrT(*,T))  & H^0((\scrT\lquot\scrC)(Q*,QT)|_{\scrT}) & 0
\Ar{1-1}{1-2}{"m^{\scrC}_T"}
\Ar{1-2}{1-3}{""}
\Ar{1-3}{1-4}{}
\Ar{2-1}{2-2}{"m^{\calC}_T"}
\Ar{2-2}{2-3}{""}
\Ar{2-3}{2-4}{}
\Ar{1-1}{2-1}{"\cong"}
\Ar{1-2}{2-2}{equal}
\Ar{1-3}{2-3}{}
\end{tikzcd}$$
in $\Mod \calT$. Thus, the third vertical morphism is an isomorphism. This shows that the natural functor $\calT/[\calC]\to H^0(\scrT\lquot\scrC)$ is fully faithful. The essential surjectivity of this functor is clear. 
\end{proof}

\subsection{Limits and approximations in connective DG-categories}

In this subsection, we recall some limits and colimits in connective DG-categories and discuss relations between approximations and exactness of $3$-term complexes in DG-quotients of connective DG-categories. We fix a universe $\bbU$ satisfying the axiom of infinity. Let $\scrT$ be a locally $K_\bbU$-projective $\bbU$-small DG-category and $\scrC$ a full DG-subcategory of $\scrT$. 

\begin{definition}
Let $\scrT$ be a DG-category. We say that $\scrT$ is \emph{additive} if the homotopy category $H^0(\scrT)$ is an additive category. 
\end{definition}

\begin{remark}
If $\scrT$ is an additive connective DG-category, then the DG-quotient $\scrT\lquot\scrC$ is also additive and connective for any full DG-subcategory $\scrC$ of $\scrT$.
\end{remark}

\begin{definition}(\cite[Definition~3.39]{che2023})
\label{def:3-term_cpx}
For a sextuple $(X,Y,Z,f,g,h)$ where $X,Y,Z\in\scrT$ and $f\in\scrT(X,Y),g\in\scrT(Y,Z)$ are closed morphisms of degree $0$ and $h\in\scrT(X,Z)$ is a morphism of degree $-1$ which satisfies $d(h)=-g\circ f$, we denote it simply by 
$$X\xrightarrow{f}Y\xrightarrow{g}Z\quad(h\colon X\to Z)$$
and call it a \emph{$3$-term complex} in $\scrT$.
\end{definition}

\begin{definition}[{\cite[\S 3.2]{che2023}}]
\label{def:cat_3-term_cpx}
Define the category $\cpx^3(\scrT)$ of $3$-term complexes in $\scrT$ as follows:
\begin{itemize}
	\item An object of $\cpx^3(\scrT)$ is a $3$-term complex in $\scrT$.
	\item A morphism from a $3$-term complex $(X,Y,Z,f,g,h)$ to another $3$-term complex $(X',Y',Z',f',g',h')$ is a homotopy equivalence class of sextuples $(a,b,c,s,s',t)$ where $a\in\scrT(X,X')^0, b\in\scrT(Y,Y')^0, c\in\scrT(Z,Z')^0$ are closed morphisms of degree $0$ and $s\in\scrT(X,Y')^{-1}, s'\in\scrT(Y,Z')^{-1}$ and $t\in\scrT(X,Z')^{-2}$ are morphisms of degree $-1$ and $-2$ respectively,
	which satisfy the following equations:
  $$d(s)=-b\circ f+f'\circ a, \quad d(s')=-c\circ g+g'\circ b, \quad d(t)=-c\circ h+h'\circ a+s'\circ f+s\circ g.$$
	The definition of homotopy equivalence is in \cite[\S 3.2]{che2023}.
\end{itemize}
\end{definition}

\begin{definition}(\cite[Definition~3.49]{che2023})
\label{def:exactness}
Let $(X,Y,Z,f,g,h)$ be a $3$-term complex in $\scrT$. 
\begin{enumerate}
	\item We say that a $3$-term complex $(X,Y,Z,f,g,h)$ is \emph{left exact} if the natural morphism 
	$(f\hat, h\hat)^t\colon X\hat\to \Cocone g\hat$
	induces isomorphisms for each $i\leq 0$:
	$$H^i(X\hat) \xrightarrow{\cong} H^i(\Cocone g\hat).$$
	In this case, we call $X$ the \emph{kernel} of $g$. 
	\item Dually, we say that a $3$-term complex $(X,Y,Z,f,g,h)$ is \emph{right exact} if the natural morphism $(g\hut, h\hut)\colon Z\hut\to \Cocone f\hut$ induces isomorphisms for each $i\leq 0$:
	$$H^i(Z\hut) \xrightarrow{\cong} H^i(\Cocone f\hut).$$
	In this case, we call $Z$ the \emph{cokernel} of $f$.
	\item If a $3$-term complex $(X,Y,Z,f,g,h)$ is both left and right exact, we say that it is \emph{exact}.
\end{enumerate}
\end{definition}

\begin{lemma}
\label{lem:unique_ker}
Let $(X,Y,Z,f,g,h)$ and $(X',Y',Z',f',g',h')$ be left exact $3$-term complexes in $\scrT$. If the following diagram exists in $\scrT$:
$$\begin{tikzcd}
Y & Z \\ Y' & Z'
\Ar{1-1}{1-2}{"g"}
\Ar{2-1}{2-2}{"g'"}
\Ar{1-1}{2-1}{"b"'}
\Ar{1-2}{2-2}{"c"}
\Ar{1-1}{2-2}{"s'", red}
\end{tikzcd}
\quad\quad d(s')=-c\circ g + g'\circ b
$$
where $b\in\scrT(Y,Y')^0$ and $c\in\scrT(Z,Z')^0$ are isomorphisms. Then there exists an isomorphism $(a,b,c,s,s',t)$ from $(X,Y,Z,f,g,h)$ to $(X',Y',Z',f',g',h')$ in $\cpx^3(\scrT)$ where $a\in\scrT(X,X')^0$ is an isomorphism. 
\end{lemma}
\begin{proof}
Since $(X',Y',Z',f',g',h')$ is left exact, there exists a morphism $(a,b,c,s,s',t)$ from $(X,Y,Z,f,g,h)$ to $(X',Y',Z',f',g',h')$ in $\cpx^3(\scrT)$ by \cite[Lemma~4.5]{che2023}. Moreover, since $b$ and $c$ are isomorphisms, by definition of left exactness, $a$ is also an isomorphism. Thus, $(a,b,c,s,s',t)$ is an isomorphism in $\cpx^3(\scrT)$ by \cite[Proposition~3.40]{che2023}.
\end{proof}

\begin{definition}[{\cite[Definition~3.14]{che2023}}]
\label{def:cartesian}
Consider the following diagram in $\scrT$:
$$\begin{tikzcd}X & Y \\ Z & W
\Ar{1-1}{1-2}{"f"}
\Ar{2-1}{2-2}{"k"}
\Ar{1-1}{2-1}{"g"'}
\Ar{1-2}{2-2}{"j"}
\Ar{1-1}{2-2}{"h", red}
\end{tikzcd}\quad\quad d(h)=j\circ f - k\circ g$$
\begin{enumerate}
	\item We say that the above diagram is \emph{cartesian} if the canonical morphism 
	$$(f,g,h)\colon X\hat\to \Cocone(-j\hat, k\hat) $$
	induces the following isomorphisms for each $i\leq 0$:
	$$H^i(X\hat) \xrightarrow{\cong} H^i(\Cocone(-j\hat, k\hat)).$$
	\item Dually, we say that the above diagram is \emph{cocartesian} if the canonical morphism
	$$(j,k,h)\colon W\hut\to \Cocone(-f\hut, g\hut) $$
	induces the following isomorphisms for each $i\leq 0$:
	$$H^i(W\hut) \xrightarrow{\cong} H^i(\Cocone(-f\hut, g\hut)).$$
	\item If the above diagram is both cartesian and cocartesian, we say that it is \emph{bicartesian}.
\end{enumerate}
\end{definition}

To discuss approximation theory in connective DG-categories, we explain the notion of approximations in connective DG-categories. Note that the notion of approximations in connective DG-categories is defined by using usual approximations in the homotopy category of a connective DG-category as follows.  

\begin{definition}
\label{def:approx}
Let $f\colon X\to Y$ be a closed morphism in $\scrT$. 
\begin{enumerate}
	\item $f$ is called a \emph{right $\scrC$-approximation} if the following map is surjective and $X\in \scrC$:
	$$f\circ -\colon \calT(-,X)|_{\calC} \to \calT(-,Y)|_{\calC}$$
	\item $f$ is called a \emph{left $\scrC$-approximation} if the following map is surjective and $Y\in \scrC$:
	$$-\circ f\colon \calT(Y,-)|_{\calC} \to \calT(X,-)|_{\calC}$$
	\item If any object $Y\in\scrT$ admits a right $\scrC$-approximation, we say that $\scrC$ is \emph{contravariantly finite} in $\scrT$.
	\item If any object $X\in\scrT$ admits a left $\scrC$-approximation, we say that $\scrC$ is \emph{covariantly finite} in $\scrT$.
	\item We say that $\scrC$ is \emph{functorially finite} in $\scrT$ if $\scrC$ is both contravariantly and covariantly finite in $\scrT$.
\end{enumerate}
\end{definition}

More generally, we can also define the notion of $\scrC$-epic and $\scrC$-monic morphisms (already defined for usual linear categories in \cite{bel1994}) in $\scrT$ as follows:

\begin{definition}
\label{def:epic_monic}
Let $f\colon X\to Y$ be a closed morphism in $\scrT$. We say that $f$ is \emph{$\scrC$-epic} if the following map is surjective:
$$f\circ -\colon \calT(-,X)|_{\calC} \to \calT(-,Y)|_{\calC}.$$
Dually, we say that $f$ is \emph{$\scrC$-monic} if the following map is surjective:
$$-\circ f\colon \calT(Y,-)|_{\calC} \to \calT(X,-)|_{\calC}.$$
\end{definition}

\begin{remark}
$\scrC$-epic morphisms are right $\scrC$-approximations if their domains are in $\scrC$. Dually, $\scrC$-monic morphisms are left $\scrC$-approximations if their codomains are in $\scrC$. 
\end{remark}

The following is a key proposition for the main result of this paper:

\begin{proposition}
\label{prop:ker_in_ideal_quot}
Consider the following diagrams in $\calT$ and $\calT\lquot\calC$ respectively:
\begin{equation}
\label{eq:diag_cart}
\begin{tikzcd}
K & T \\ X & Y
\Ar{1-1}{1-2}{"k"}
\Ar{2-1}{2-2}{"f"'}
\Ar{1-1}{2-1}{"i"'}
\Ar{1-2}{2-2}{"\pi"}
\Ar{1-1}{2-2}{"h", red}
\end{tikzcd}\quad\quad\quad
\begin{tikzcd}
QK & QT \\ QX & QY
\Ar{1-1}{1-2}{"Qk"}
\Ar{2-1}{2-2}{"Qf"'}
\Ar{1-1}{2-1}{"Qi"'}
\Ar{1-2}{2-2}{"Q\pi"}
\Ar{1-1}{2-2}{"Qh", red}
\end{tikzcd}
\end{equation}
If the left diagram is cartesian in $\scrT$ and the following induced morphism is an epimorphism in $\Mod \calC$:
$$(f\circ -, \pi\circ -)|_{\calC}\colon \calT(*,X)|_{\calC}\oplus \calT(*,T)|_{\calC} \to \calT(*,Y)|_{\calC},$$
then the right-hand diagram is also cartesian in $\scrT\lquot\scrC$. In particular, if $T\in\scrC$, $QK$ is a kernel of $Qf$ in $\scrT\lquot\scrC$.
\end{proposition}
\begin{proof}
First, we note that the canonical morphism
\begin{equation}
\label{eq:isom_cart}
(k\hat, i\hat, h\hat)^t|_{\scrC}\colon K\hat|_{\scrC} \to \Cocone(-{\pi}\hat|_{\scrC}, f\hat|_{\scrC})
\end{equation}
is an isomorphism in $\calD(\scrC)$ because 
$H^1\bigl(\Cocone(-{\pi}\hat|_{\scrC}, f\hat|_{\scrC})\bigr)=0$. 
For brevity, we denote $\Phi := \Cocone(-{\pi}\hat, f\hat)$. We have the following diagram in $\calD(\scrT)$:
$$
\begin{tikzcd}
K\hat|_{\scrC}\lox_{\scrC}\scrT|_{\scrC\op}
&
K\hat
&
QK\hat|_{\scrT}
\Ar{1-1}{1-2}{"\epsilon_{K\hat}"}
\Ar{1-2}{1-3}{}
\Ar{1-1}{2-1}{"\bigl((k\hat\cma i\hat\cma h\hat)^t|_{\scrC}\bigr)\lox_{\scrC}\scrT|_{\scrC\op}"'}
\Ar{1-2}{2-2}{"(k\hat\cma i\hat\cma h\hat)^t"}
\Ar{1-3}{2-3}{"\bigl((Qk\hat\cma Qi\hat\cma Qh\hat)^t\bigr)|_{\scrT}"}
\\
\Phi|_{\scrC}\lox_{\scrC}\scrT|_{\scrC\op}
&
\Phi
&
\bigl(\Cocone(-{Q\pi}\hat, Qf\hat)\bigr)|_{\scrT}
\Ar{2-1}{2-2}{"\epsilon_{\Phi}"}
\Ar{2-2}{2-3}{}
\Ar{2-1}{3-1}{}
\Ar{2-2}{3-2}{}
\Ar{2-3}{3-3}{}
\\
\Cone\bigl((k\hat,i\hat,h\hat)^t|_{\scrC}\lox_{\scrC}\scrT|_{\scrC\op}\bigr)
&
\Cone(k\hat,i\hat,h\hat)^t
&
\Cone\bigl((Qk\hat,Qi\hat,Qh\hat)^t|_{\scrT}\bigr)
\Ar{3-1}{3-2}{}
\Ar{3-2}{3-3}{}
\end{tikzcd}
$$
where each row and each column are triangles in $\calD(\scrT)$. Since the morphism \eqref{eq:isom_cart} is an isomorphism in $\calD(\scrC)$, we have 
$$\Cone\bigl((k\hat,i\hat,h\hat)^t|_{\scrC}\lox_{\scrC}\scrT|_{\scrC\op}\bigr)=0$$
and hence the following morphism is an isomorphism in $\calD(\scrT)$:
$$\Cone (k\hat, i\hat, h\hat)^t \to \Cone\bigl((Qk\hat,Qi\hat,Qh\hat)^t|_{\scrT}\bigr).$$
Since we assume that the $H^i\Cone (k\hat, i\hat, h\hat)^t=0$ for $i\leq 0$, we have the following isomorphisms in $\Mod\calT$ for each $i\leq 0$:
$$H^i(QK\hat|_{\scrT}) \cong H^i\bigl(\Cocone(-{Q\pi}\hat, Qf\hat)|_{\scrT}\bigr).$$
Because $(-)|_{\calT}\colon\calD(\scrT\lquot\scrC)\to \calD(\scrT)$ preserves standard $t$-structures and is fully faithful, we have the following isomorphisms in $\Mod(\calT/[\calC])$ for each $i\leq 0$:
$$H^i(QK\hat)\cong H^i\Cocone(-{Q\pi}\hat, Qf\hat).$$
This shows that the right-hand diagram in \eqref{eq:diag_cart} is cartesian in $\scrT\lquot\scrC$. 
\end{proof}

\begin{definition}
\label{def:fin_comp}
Let $\scrT$ be a connective and additive DG-category. We say that $\scrT$ is \emph{finitely complete} if any closed morphism of degree $0$ has a kernel. Dually, we say that $\scrT$ is \emph{finitely cocomplete} if any closed morphism of degree $0$ has a cokernel. 
\end{definition}

\begin{cor}
\label{cor:ideal_quot_fin_comp}
Let $\scrT$ be a finitely complete DG-category and $\scrC$ a contravariantly finite full DG-subcategory of $\scrT$. Then the DG-quotient $\scrT\lquot\scrC$ is also finitely complete. 
\end{cor}
\begin{proof}
This is a direct consequence of Proposition~\ref{prop:ker_in_ideal_quot}.
\end{proof}

Although not directly related to our main results, the above corollary can be regarded as a DG-enhancement of the theorem of Beligianis, Marmaridis and Reiten \cite{bel1994,bel2007}. Moreover, this proposition can also be used in the proof of results on left triangulated categories by Chen. We shall discuss this. 

\begin{fact}[{\cite[Theorem~2.12]{bel1994}, \cite[Chapter~II, 1]{bel2007}}]
\label{thm:BM}
Let $\calT$ be an additive category with kernels and $\calC$ a contravariantly finite full subcategory of $\calT$. Then the quotient category $\calT/[\calC]$ naturally admits a left triangulated structure. Moreover, in addition, if $\calT$ admits cokernels and $\calC$ is functorially finite in $\calT$, then the quotient category $\calT/[\calC]$ naturally admits a pretriangulated structure.
\end{fact}

The above theorem also holds more generally for algebraic left triangulated categories replacing additive categories with kernels. To explain this, we need the following definition of algebraic left triangulated categories. 

The following proposition shows that finitely complete DG-categories provide a DG-enhancement of left triangulated categories:
\begin{fact}[{\cite[Proposition~2.17]{moc2025}}]
\label{prop:htpy_cat_lef}
Let $\scrT$ be a finitely complete DG-category. Then the homotopy category $H^0(\scrT)$ admits a canonical left triangulated structure. In addition, if $\scrT$ is also finitely cocomplete, then $H^0(\scrT)$ admits a canonical pretriangulated structure.
\end{fact}

\begin{definition}[{\cite{chen2026enhancedlefttriangulatedcategories}}]
\label{def:alg_left_tri}
Let $(\calT, \Omega, \Delta)$ be a left triangulated category. We say that it is \emph{algebraic} if there exists a finitely complete DG-category $\scrT$ such that $H^0(\scrT)\simeq \calT$ as left triangulated categories. In this case, we say that $\scrT$ is a \emph{DG-enhancement} of $\calT$. Similarly, we can also define the notion of algebraic pretriangulated categories and their DG-enhancements by using finitely complete and finitely cocomplete DG-categories respectively.
\end{definition}

The following is a generalization of Fact~\ref{thm:BM} for algebraic left triangulated categories:
\begin{cor}
\label{cor:alg_enhance}
Let $\calT$ be an algebraic left triangulated category and $\calC$ a contravariantly finite full subcategory of $\calT$. Then the quotient category $\calT/[\calC]$ naturally admits a left triangulated structure. Moreover, it has a DG-enhancement given by the DG-quotient of a DG-enhancement of $\calT$ by a DG-enhancement of $\calC$. 
\end{cor}
\begin{proof}
We already have two left triangulated structures on $\calT/[\calC]$. One is given by \cite[Proposition~2.17]{moc2025} and Corollary~\ref{cor:ideal_quot_fin_comp}. The other is given by the canonical construction of a left triangulated structure on $\calT/[\calC]$ (see \cite{bel1994,bel2007}). Checking that these two left triangulated structures coincide is straightforward, and we omit the details.
\end{proof}

\begin{cor}
\label{cor:alg_enhance_ideal_quot}
Let $\calT$ be an algebraic pretriangulated category and $\calC$ a functorially finite full subcategory of $\calT$. Then the quotient category $\calT/[\calC]$ naturally admits a pretriangulated structure. Moreover, it has a DG-enhancement given by the DG-quotient of a DG-enhancement of $\calT$ by a DG-enhancement of $\calC$.
\end{cor}
\begin{proof}
This is similar to the proof of Corollary~\ref{cor:alg_enhance}, and we omit the details.
\end{proof}

We can also apply the above results to the setting of exact DG-categories. The following theorem is already known by \cite{chen2026enhancedlefttriangulatedcategories}, and we give another proof by using Proposition~\ref{prop:ker_in_ideal_quot}.

\begin{theorem}(\cite[Theorem~3.1]{chen2026enhancedlefttriangulatedcategories})
Let $(\scrB,\calS)$ be a connective exact DG-category with enough projectives. Denote by $\scrP$ the full DG-subcategory of $\scrB$ consisting of projective objects. Then the DG-quotient $\scrB\lquot\scrP$ is a finitely complete DG-category.
\end{theorem}
\begin{proof}
We give another proof of this theorem by using Proposition~\ref{prop:ker_in_ideal_quot}. Let $f\colon X\to Y$ be a closed morphism of degree $0$ in $\scrB$. Since $\scrB$ has enough projectives, there exists a right $\scrP$-approximation $\pi\colon P\to Y$ of $Y$. Then we can take a cartesian diagram in $\calB$ as follows by using the pullback axiom of exact DG-categories:
$$\begin{tikzcd}
K & P \\ X & Y
\Ar{1-1}{1-2}{"k"}
\Ar{2-1}{2-2}{"f"'}
\Ar{1-1}{2-1}{"i"'}
\Ar{1-2}{2-2}{"\pi"}
\Ar{1-1}{2-2}{"h", red}
\end{tikzcd}$$
By Proposition~\ref{prop:ker_in_ideal_quot}, the image of the above diagram in $\scrB\lquot\scrP$ is also cartesian. Thus, $QK$ is a kernel of $Qf$ in $\scrB\lquot\scrP$. This shows that $\scrB\lquot\scrP$ is finitely complete.
\end{proof}

\subsection[Stable DG-categories and Abelian d-truncated DG-categories]{Stable DG-categories and Abelian \texorpdfstring{$d$}{d}-truncated DG-categories}

We fix an additive connective DG-category $\scrT$.

\begin{definition}(\cite[Definition~6.1]{che2024})
\label{def:stable_DG}
Let $\scrT$ be an additive connective DG-category. We say that $\scrT$ is \emph{stable} if the following conditions are satisfied:
\begin{enumerate}
	\item[(St1)] $\scrT$ is finitely complete and finitely cocomplete.
	\item[(St2)] Any left exact $3$-term complex is also right exact. 
	\item[(St2)$\op$] Any right exact $3$-term complex is also left exact. 
\end{enumerate}
\end{definition}

\begin{example}
\label{ex:stable_DG}
Let $\scrT'$ be a pretriangulated DG-category. Then $\tau^{\leq 0}\scrT'$ is stable by \cite[Example~6.2]{che2024}. Conversely, if $\scrT$ is a stable DG-category, then it is a connective cover of some pretriangulated DG-category by \cite[Proposition~3.26]{che2024b}. Thus, stable DG-categories provide connective DG-enhancements of algebraic triangulated categories.
\end{example}

\begin{theorem}(\cite[Theorem~6.5]{che2024})
Let $\scrT$ be a stable DG-category. Then the homotopy category $H^0(\scrT)$ admits a triangulated structure. 
\end{theorem}

For the rest, we fix a stable DG-category $\scrT$. We may denote the homotopy category $H^0(\scrT)$ by $\calT$. We denote the shift functor of $\calT$ by $[1]$.

\begin{definition}(\cite[Section~3]{IY2008})
\label{def:cluster_tilting_subcat}
Let $\scrT$ be a stable DG-category and $\scrC$ a full DG-subcategory of $\scrT$. We say that $\scrC$ is a \emph{$(d+1)$-cluster tilting DG-subcategory} of $\scrT$ if the following conditions are satisfied:
\begin{enumerate}
	\item $\calC$ is functorially finite in $\calT$.
	\item For any $X\in\scrT$, we have $X\in\scrC$ if and only if $\calT(X,\calC[i])=0$ for any $1\leq i\leq d$.
	\item For any $X\in\scrT$, we have $X\in\scrC$ if and only if $\calT(\calC,X[i])=0$ for any $1\leq i\leq d$.
\end{enumerate}
\end{definition}

\begin{remark}
The definition above of a $(d+1)$-cluster tilting DG-subcategory in a stable DG-category is equivalent to $\calC\subset\calT$ being a $(d+1)$-cluster tilting subcategory in the usual sense.
\end{remark}

\begin{definition}
\label{def:ast_sub}
Let $\scrX,\scrY$ be full subcategories of $\calT$. We denote by $\scrX*\scrY$ the full subcategory of $\scrT$ consisting of objects $M\in\scrT$ such that there exists a triangle with $X\in\scrX$ and $Y\in\scrY$:
$$X\to M\to Y\to X[1]$$
Note that this operation is associative, i.e., we have the following equality for any full subcategories $\scrX,\scrY,\scrZ$ of $\calT$:
$$(\scrX*\scrY)*\scrZ=\scrX*(\scrY*\scrZ).$$ 
\end{definition}

The following theorem is useful.

\begin{fact}(\cite[Theorem~3.1]{IY2008})
\label{thm:splicing}
Let $\scrT$ be a stable DG-category and $\scrC$ a $(d+1)$-cluster tilting DG-subcategory of $\scrT$. Then we have the following equality of subcategories of $\calT$:
$$\scrT=\scrC*\scrC[1]*\cdots*\scrC[d].$$
\end{fact}

\begin{remark}
\label{rem:shihted}
By shifting the above equation, we also have the following equation as subcategories of $\calT$:
$$\scrT=\scrC[-d]*\cdots*\scrC[-1]*\scrC.$$
\end{remark}

% \color{blue}
% \begin{lemma}
% Let $\scrC$ be a $(d+1)$-cluster tilting DG-subcategory of $\scrT$. Consider the following $3$-term complex in $\scrT$:
% $$X\xrightarrow{f} Y\xrightarrow{g} Z \quad (h\colon X\to Z)$$
% where $Y\in\scrC$ and $g$ is a right $\scrC$-approximation of $Z$. 
% \end{lemma}
% \normalcolor

% \subsection[Abelian d-truncated DG-categories]{Abelian \texorpdfstring{$d$}{d}-truncated DG-categories}
% In this subsection, we fix a connective DG-category $\scrA$. 
Next, we explain the notion of abelian $d$-truncated DG-categories. The notion of abelian $d$-truncated DG-categories is a $d$-dimensional analogue of abelian DG-categories. We fix a connective DG-category $\scrA$ for the rest of this subsection.

\begin{definition}
\label{def:d-truncated}
We say that $\scrA$ is \emph{$d$-truncated} if $H^i(\scrA(X,Y))=0$ for any $X,Y\in\scrA$ and $i\leq -d$. 
\end{definition}

\begin{definition}
\label{def:n-mono}
Let $f\colon X\to Y$ be a closed morphism of degree $0$ in $\scrA$. We say that $f$ is an \emph{$n$-monomorphism} if the following conditions hold.
\begin{enumerate}
	\item The following morphism is a monomorphism in $\Mod \calA$:
	$$H^{-n+1}(\scrA(*,X))\to H^{-n+1}(\scrA(*,Y))$$
	\item The following morphism is an isomorphism in $\Mod \calA$ for each $i\leq -n$:
	$$H^i(\scrA(*,X))\to H^i(\scrA(*,Y))$$
\end{enumerate}
Dually, we say that $f$ is an \emph{$n$-epimorphism} if the following conditions hold.
\begin{enumerate}
	\item The following morphism is a monomorphism in $\Mod \calA$:
	$$H^{-n+1}(\scrA(Y,*))\to H^{-n+1}(\scrA(X,*))$$
	\item The following morphism is an isomorphism in $\Mod \calA$ for each $i\leq -n$:
	$$H^i(\scrA(Y,*))\to H^i(\scrA(X,*))$$
\end{enumerate}
\end{definition}

\begin{definition}
\label{def:loop_obj}
Let $X\in\scrA$ be an object. We say that $Y\in\scrA$ is a \emph{loop object} of $X$ and denote it by $\Omega X$ if there exists an isomorphism $\scrA(*,Y)\cong \tau_{\leq 0}(\scrA(*,X)[-1])$ in $\calD(\calA)$. Dually, we say that $Z\in\scrA$ is a \emph{suspension object} of $X$ and denote it by $\Sigma X$ if there exists an isomorphism $\scrA(Z,*)\cong \tau_{\leq 0}(\scrA(X,*)[-1])$ in $\calD(\calA\op)$.
\end{definition}
% \textcolor{blue}{Include the definition of loop object in connective DG-categories.}

\begin{remark}
\label{rem:loop_obj}
By choosing a loop object $\Omega X$ for each $X\in\scrA$, we can define a loop functor $\Omega\colon\calA\to \calA$ on the homotopy category $\calA$ of $\scrA$ by sending $X$ to $\Omega X$. Dually, we also have a suspension functor $\Sigma\colon\calA\to \calA$ by sending $X$ to $\Sigma X$. It is already known that the loop functor $\Omega$ is right adjoint to the suspension functor $\Sigma$ (see \cite[\S 2.1]{moc2025}). Note that this is different from the syzygy functor. We do not use the notation $\Omega$ for the syzygy functor in this paper.
\end{remark}

\begin{remark}
\label{rem:d-trunc_loop}
Assume that $\scrA$ has loop objects. Then $\scrA$ is $d$-truncated if and only if $\Omega^d X=0$ holds for any $X\in\scrA$ if and only if $\Sigma^d X=0$ holds for any $X\in\scrA$.
\end{remark}

\begin{lemma}(\cite[Proposition~3.8]{moc2025})
\label{lem:chara_d_mono_abel}
Consider the following left exact $3$-term complex in $\scrA$:
$$K\xrightarrow{k} X\xrightarrow{f} Y \quad (h\colon K\to Y)$$
Then $f$ is a $n$-monomorphism if and only if $\Omega^{n-1} K=0$ holds in $\calA$. 
\end{lemma}

\begin{lemma}(\cite[Proposition~3.11]{moc2025})
\label{lem:ker_id_d_mono}
Let $\scrA$ be a $d$-truncated DG-category and consider the following left exact $3$-term complex in $\scrA$:
$$K\xrightarrow{k} X\xrightarrow{f} Y \quad (h\colon K\to Y)$$
Then the morphism $k\colon K\to X$ is a $d$-monomorphism. 
\end{lemma}

\begin{definition}(\cite[Definition~3.12]{moc2025})
\label{def:abelian_d_DG}
We say that $\scrA$ is an \emph{abelian $d$-truncated DG-category} if $\scrA$ is an additive $d$-truncated connective DG-category and the following conditions are satisfied:
\begin{enumerate}
	\item $\scrA$ is finitely complete and finitely cocomplete.
	\item Any right exact $3$-term complex $(X,Y,Z,f,g,h)$ in $\scrA$ with $f$ being a $d$-monomorphism is also left exact. 
	\item Any left exact $3$-term complex $(X,Y,Z,f,g,h)$ in $\scrA$ with $g$ being a $d$-epimorphism is also right exact.
\end{enumerate}
\end{definition}

\section{Main results}

\subsection{From general triangulated categories}

In this section, we fix the following:
\begin{itemize}
	\item a $\bbV$-small stable DG-category $\scrT$ and assume that $\scrT$ is locally $\bbU$-small and locally $K$-projective. In addition, we also assume that $\scrT$ has essentially $\bbU$-small homotopy category $\calT$.
	\item a $(d+1)$-cluster tilting DG-subcategory $\scrC$ of $\scrT$.
\end{itemize}
Note that since $\scrT$ is locally $\bbU$-small and its homotopy category $\calT$ is essentially $\bbU$-small, the DG-quotient $\scrT\lquot\scrC$ can be taken in the universe $\bbV$ and it is also locally $\bbU$-small and has essentially $\bbU$-small homotopy category. 

We denote $\scrC_i^j:=\scrC[i]*\cdots*\scrC[j]$ for any integers $i\leq j$. By Fact~\ref{thm:splicing}, we have $\scrT=\scrC_0^d$. By Remark~\ref{rem:shihted}, we also have $\scrT=\scrC_{-d}^0$.

\begin{lemma}
\label{lem:exact_in_ideal_quot}
Assume that the following $3$-term complex in $\scrT$ is left exact:
$$X\xrightarrow{f} Y\xrightarrow{g} Z \quad (h\colon X\to Z).$$
If $g$ is $\scrC$-epic, then its image in $\scrT\lquot\scrC$ is also left exact. Dually, if the above $3$-term complex is right exact and $f$ is $\scrC$-monic, then its image in $\scrT\lquot\scrC$ is also right exact.
\end{lemma}
\begin{proof}
It is immediate from Proposition~\ref{prop:ker_in_ideal_quot}.
\end{proof}

\begin{cor}
\label{cor:loop}
Take a right $\scrC$-approximation $g\colon Y\to Z$ of $Z$ and exact $3$-term complex in $\scrT$:
$$K\xrightarrow{f} Y\xrightarrow{g}Z \quad (h\colon K\to Z)$$
Then $QK$ defines a loop object of $QZ$ in $\scrT\lquot\scrC$. Dually, take a left $\scrC$-approximation $f\colon X\to Y$ of $X$ and exact $3$-term complex in $\scrT$:
$$X\xrightarrow{f} Y\xrightarrow{g} C \quad (h\colon X\to C)$$
Then $QC$ defines a suspension object of $QX$ in $\scrT\lquot\scrC$.
\end{cor}

\begin{proposition}
\label{prop:ideal_quot_d_trun}
The DG-quotient $\scrT\lquot\scrC$ is $d$-truncated.
\end{proposition}
\begin{proof}
Since $\scrC$ is a $(d+1)$-cluster tilting DG-subcategory of $\scrT$, we can take exact $3$-term complexes in $\scrT$ for any $T\in\scrT$:
$$K_i\xrightarrow{f_i} C_i \xrightarrow{g_i} K_{i-1} \quad (h_i\colon K_i\to K_{i-1})$$ 
where $C_i\in\scrC$ and $g_i$ is a right $\scrC$-approximation for each $1\leq i\leq d$, $K_0=T$ and $K_d\in\scrC$ by Fact~\ref{thm:splicing}.

By Corollary~\ref{cor:loop}, we have $QK_i\cong \Omega^i (QT)$ for each $1\leq i\leq d$. Since $\Omega^d(QT)\cong QK_d=0$ in $\scrT\lquot\scrC$, we conclude that $H^i((\scrT\lquot\scrC)(*,QT))=0$ for any $i\leq -d$ and any $T\in\scrT$. Thus, $\scrT\lquot\scrC$ is $d$-truncated.
\end{proof}

% \begin{proposition}
% \label{prop:chara_d_mono}
% Let $f\colon X\to Y$ be a closed morphism of degree $0$ in $\scrT$ and consider the following exact $3$-term complex in $\scrT$:
% \begin{equation}
% \label{eq:3-term_mono}
% K\xrightarrow{f} Y\xrightarrow{g} Z \quad (h\colon K\to Z)
% \end{equation}
% Then $Qf\colon QX\to QY$ is a $d$-monomorphism in $\scrT\lquot\scrC$ if and only if $\defe_{\scrC}g=0$ holds in $\Mod \calC$. 
% \end{proposition}
% \begin{proof}
% If $\defe_{\scrC}g=0$ holds in $\Mod \calC$, then the image of the $3$-term complex \eqref{eq:3-term_mono} in $\scrT\lquot\scrC$ is also left exact by Lemma~\ref{lem:exact_in_ideal_quot}. Since left side morphism in left exact $3$-term complex of $d$-truncated DG-category is always a $d$-monomorphism, we conclude that $Qf$ is a $d$-monomorphism in $\scrT\lquot\scrC$. 

% We show the converse. Assume that $Qf$ is a $d$-monomorphism in $\scrT\lquot\scrC$. 
% \end{proof}

\begin{lemma}
\label{lem:chara_n_th_loop}
For any object $T\in\scrT$ and each $1\leq n\leq d$, a $n$-th loop object $\Omega^{n} (QT)$ of $QT$ in $\scrT\lquot\scrC$ is zero if and only if $T$ is in the sub DG-category $\scrC_0^n$.
\end{lemma}
\begin{proof}
Assume that $\Omega^{n} (QT)=0$ holds in $\scrT\lquot\scrC$. Since we have the following exact $3$-term complexes in $\scrT$ for each $1\leq i\leq n$ by assumption:
\begin{equation}
\label{eq:3-term_loop}
K_i\xrightarrow{f_i} C_i \xrightarrow{g_i} K_{i-1} \quad (h_i\colon K_i\to K_{i-1})
\end{equation} 
where $C_i\in\scrC$ and $g_i$ is a right $\scrC$-approximation for each $1\leq i\leq n$ and $K_0=T$ by Fact~\ref{thm:splicing}, we have an isomorphism $QK_n\cong \Omega^n (QT)=0$ in $\scrT\lquot\scrC$ by Corollary~\ref{cor:loop}. Since $\scrC$ is closed under direct summands, this implies that $K_n$ is in $\scrC$ and hence $T=K_0$ is in $\scrC_0^n$.

Conversely, assume that $T$ is in $\scrC_0^n$. Then we can construct exact $3$-term complexes in $\scrT$ for each $1\leq i\leq n$ as above \eqref{eq:3-term_loop} where $C_i\in\scrC$ for each $1\leq i\leq n$, $K_0=T$ and $K_n\in\scrC$. Thus, we have $\Omega^n (QT)\cong QK_n=0$ in $\scrT\lquot\scrC$ by Corollary~\ref{cor:loop}.
\end{proof}

\begin{lemma}
\label{lem:chara_d_mono_first}
Let $n$ be an integer with $1\leq n\leq d$. Consider the following cartesian diagram in $\scrT$:
$$\begin{tikzcd}
K & C \\ X & Y
\Ar{1-1}{1-2}{"k"}
\Ar{2-1}{2-2}{"f"'}
\Ar{1-1}{2-1}{"i"'}
\Ar{1-2}{2-2}{"\pi"}
\Ar{1-1}{2-2}{"h", red}
\end{tikzcd}$$
where $C\in\scrC$ and the induced morphism
$$(f\circ-,\pi\circ-)\colon \calT(*,X)|_{\calC}\oplus \calT(*,C)|_{\calC} \to \calT(*,Y)|_{\calC}$$
is an epimorphism in $\Mod \calC$.
Then the following conditions are equivalent:
\begin{enumerate}
	\item The morphism $Qf\colon QX\to QY$ is a $n$-monomorphism in $\scrT\lquot\scrC$.
	\item $K$ is in the sub DG-category $\scrC_0^{n-1}$.
\end{enumerate}
\end{lemma}
\begin{proof}
Note that the image of the above cartesian diagram in $\scrT\lquot\scrC$ is also cartesian by Proposition~\ref{prop:ker_in_ideal_quot}. Thus, $QK$ is a kernel of $Qf$ in $\scrT\lquot\scrC$.
By Lemma~\ref{lem:chara_d_mono_abel}, $Qf$ is a $n$-monomorphism in $\scrT\lquot\scrC$ if and only if $\Omega^{n-1} (QK)=0$ holds in $\scrT\lquot\scrC$. By Lemma~\ref{lem:chara_n_th_loop}, this is equivalent to $K$ being in the sub DG-category $\scrC_0^{n-1}$.
\end{proof}

\begin{lemma}
\label{lem:d_mono_if_cocone}
Let $f\colon X\to Y$ be a closed morphism in $\scrT$ and $1\leq n\leq d$. If $\Cocone(f)$ is in $\scrC_0^{n-1}$, then $Qf$ is a $n$-monomorphism in $\scrT\lquot\scrC$.
\end{lemma}
\begin{proof}
Take a right $\scrC$-approximation $\pi\colon C\to Y$ of $Y$ and the following cartesian diagram in $\scrT$:
$$\begin{tikzcd}
K & C \\ X & Y
\Ar{1-1}{1-2}{"k"}
\Ar{2-1}{2-2}{"f"'}
\Ar{1-1}{2-1}{"i"'}
\Ar{1-2}{2-2}{"\pi"}
\Ar{1-1}{2-2}{"h", red}
\end{tikzcd}$$
Then we have the following diagram in the homotopy category $\calT$:
$$\begin{tikzcd}
\Cocone(f) & K & C & \Cocone(f)[1] \\
\Cocone(f) & X & Y & \Cocone(f)[1]
\Ar{1-1}{1-2}{}
\Ar{1-2}{1-3}{"k"}
\Ar{1-3}{1-4}{}
\Ar{2-1}{2-2}{}
\Ar{2-2}{2-3}{"f"}
\Ar{2-3}{2-4}{}
\Ar{1-1}{2-1}{equal}
\Ar{1-2}{2-2}{"i"}
\Ar{1-3}{2-3}{"f"}
\Ar{1-4}{2-4}{equal}
\end{tikzcd}$$
Since $\Cocone(f)$ is in $\scrC_0^{n-1}$ and $C\in \scrC$, we have $K\in \scrC_0^{n-1}$. Thus, by Lemma~\ref{lem:chara_d_mono_first}, $Qf$ is a $n$-monomorphism in $\scrT\lquot\scrC$.
\end{proof}

\begin{proposition}
\label{prop:d_mono_chara_main}
Consider the following bicartesian diagram in $\scrT$:
$$\begin{tikzcd}
X & C \\ Y & Z
\Ar{1-1}{1-2}{"\iota"}
\Ar{2-1}{2-2}{"g"'}
\Ar{1-1}{2-1}{"f"'}
\Ar{1-2}{2-2}{"\pi"}
\Ar{1-1}{2-2}{"h", red}
\end{tikzcd}$$
where $C\in\scrC$ and let 
$$X\xrightarrow{(f, -\iota)^t} Y\oplus C \xrightarrow{(g, \pi)} Z \xrightarrow{\delta} X[1]$$
be the associated triangle in $\calT$. Then the following conditions are equivalent:
\begin{enumerate}
	\item $Qf$ is a $d$-monomorphism in $\scrT\lquot\scrC$.
	\item The morphism $\delta[-1]\colon Z[-1]\to X$ factors through an object in $\calC_0^{d-1}$ in $\calT$.
	\item The induced morphism
	$$(g\circ-, \pi\circ-)\colon \calT(*,Y)|_{\calC}\oplus \calT(*,C)|_{\calC} \to \calT(*,Z)|_{\calC}$$
	is an epimorphism in $\Mod \calC$.
\end{enumerate}
\end{proposition}
\begin{proof}
(3) $\Rightarrow$ (1): First, the image of the above bicartesian diagram in $\scrT\lquot\scrC$ is also cartesian by Proposition~\ref{prop:ker_in_ideal_quot}. Thus, we have a left exact $3$-term complex in $\scrT\lquot\scrC$:
$$QX\xrightarrow{Qf} QY \xrightarrow{Qg} QZ \quad (Qh-Q(\pi)\circ h_{QC}\circ Q(\iota)\colon QX\to QZ)$$
where $h_{QC}\colon QC\to QC$ is a degree $-1$ morphism such that $dh_{QC}=\id_{QC}$. Since $\scrT\lquot\scrC$ is a $d$-truncated DG-category by Proposition~\ref{prop:ideal_quot_d_trun}, $Qf$ is a $d$-monomorphism in $\scrT\lquot\scrC$.

(1) $\Rightarrow$ (2): Take a right $\scrC$-approximation $\pi' \colon C'\to Y$ of $Y$ and a cartesian diagram in $\scrT$:
$$\begin{tikzcd}
K & C' \\ X & Y
\Ar{1-1}{1-2}{"\iota'"}
\Ar{2-1}{2-2}{"f"'}
\Ar{1-1}{2-1}{"k"'}
\Ar{1-2}{2-2}{"\pi'"}
\Ar{1-1}{2-2}{"h'", red}
\end{tikzcd}$$
Then $K\in \scrC_0^{d-1}$ by Lemma~\ref{lem:chara_d_mono_first}. We will show that the morphism $\delta[-1]\colon Z[-1]\to X$ factors through an object $K$ in $\calT$. We have the following diagram in $\calT$:
$$\begin{tikzcd}[column sep=30]
Z[-1] & X & Y\oplus C & Z \\
\Cone(f)[-1] & X & Y & \Cone(f)
\Ar{1-1}{1-2}{"{\delta[-1]}"}
\Ar{1-2}{1-3}{"{(f, -\iota)^t}"}
\Ar{1-3}{1-4}{"{(g, \pi)}"}
\Ar{2-1}{2-2}{}
\Ar{2-2}{2-3}{"f"}
\Ar{2-3}{2-4}{}
\Ar{1-1}{2-1}{}
\Ar{1-2}{2-2}{equal}
\Ar{1-3}{2-3}{"{(1,0)^t}"}
\Ar{1-4}{2-4}{}
\end{tikzcd}$$
Thus, it is enough to show that the morphism $\Cone(f)[-1]\to X$ factors through $K$. It immediately follows from the following diagram in $\calT$:
$$\begin{tikzcd}
\Cone(f)[-1] & K & C' & \Cone(f) \\
\Cone(f)[-1] & X & Y & \Cone(f)
\Ar{1-1}{1-2}{}
\Ar{1-2}{1-3}{"\iota'"}
\Ar{1-3}{1-4}{}
\Ar{2-1}{2-2}{}
\Ar{2-2}{2-3}{"f"}
\Ar{2-3}{2-4}{}
\Ar{1-1}{2-1}{equal}
\Ar{1-2}{2-2}{"k"}
\Ar{1-3}{2-3}{"\pi'"}
\Ar{1-4}{2-4}{equal}
\end{tikzcd}$$
which is obtained by using the octahedral axiom in $\calT$. Then $K\in\scrC_0^{d-1}$ by Lemma~\ref{lem:chara_d_mono_first}, and hence
the morphism $\Cone(f)[-1]\to X$ factors through $K$ and hence $\delta[-1]\colon Z[-1]\to X$ factors through $K$. 

(2) $\Rightarrow$ (3): Assume that the morphism $\delta[-1]\colon Z[-1]\to X$ factors through an object $K$ in $\calC_0^{d-1}$. Then, we obtain the following exact sequence in $\Mod \calT$:
$$\calT(*,Y)\oplus \calT(*,C)\xrightarrow{(g\circ-, \pi\circ-)}\calT(*,Z)\xrightarrow{\delta\circ-}\calT(*,X[1])$$
Since $\delta$ factors through an object $K$ in $\calC[1]*\cdots*\calC[d]$, we have $(\delta\circ-)|_{\calC} =0$ in $\Mod \calC$. Thus, the morphism 
$$(g\circ-, \pi\circ-)\colon \calT(*,Y)|_{\calC}\oplus \calT(*,C)|_{\calC} \to \calT(*,Z)|_{\calC}$$
is an epimorphism in $\Mod \calC$.
\end{proof}

\begin{proposition}
\label{prop:d_epi_chara_main}
Consider the following bicartesian diagram in $\scrT$:
$$\begin{tikzcd}
X & C \\ Y & Z
\Ar{1-1}{1-2}{"\iota"}
\Ar{2-1}{2-2}{"g"'}
\Ar{1-1}{2-1}{"f"'}
\Ar{1-2}{2-2}{"\pi"}
\Ar{1-1}{2-2}{"h", red}
\end{tikzcd}$$
where $C\in\scrC$ and let 
$$X\xrightarrow{(f, -\iota)^t} Y\oplus C \xrightarrow{(g, \pi)} Z \xrightarrow{\delta} X[1]$$
be the associated triangle in $\calT$. Then the following conditions are equivalent:
\begin{enumerate}
	\item $Qg$ is a $d$-epimorphism in $\scrT\lquot\scrC$.
	\item The morphism $\delta\colon Z\to X[1]$ factors through an object in $\calC_{-d+1}^0$ in $\calT$.
	\item The induced morphism
	$$(-\circ f, -\circ \iota)\colon \calT(Y,*)|_{\calC\op}\oplus \calT(C,*)|_{\calC\op} \to \calT(X,*)|_{\calC\op}$$
	is an epimorphism in $\Mod \calC\op$.
\end{enumerate}
\end{proposition}
\begin{proof}
The proof is dual to the proof of Proposition~\ref{prop:d_mono_chara_main}.
\end{proof}

\begin{cor}
\label{cor_exact_in_ideal}
Consider the following bicartesian diagram in $\scrT$:
$$\begin{tikzcd}
X & C \\ Y & Z
\Ar{1-1}{1-2}{"\iota"}
\Ar{2-1}{2-2}{"g"'}
\Ar{1-1}{2-1}{"f"'}
\Ar{1-2}{2-2}{"\pi"}
\Ar{1-1}{2-2}{"h", red}
\end{tikzcd}$$
where $C\in\scrC$. 
If $Qf$ is a $d$-monomorphism in $\scrT\lquot\scrC$ and $\iota$ is a left $\scrC$-approximation of $X$ in $\scrT$, then the image of the above bicartesian diagram in $\scrT\lquot\scrC$ is also bicartesian. In particular, we have the following exact $3$-term complex in $\scrT\lquot\scrC$:
\begin{align*}
QX & \xrightarrow{Qf} QY \xrightarrow{Qg} QZ \quad (Qh-Q(\pi)\circ h_{QC}\circ Q(\iota)\colon QX\to QZ)
\end{align*}
where $h_{QC}\colon QC\to QC$ is a degree $-1$ morphism such that $dh_{QC}=\id_{QC}$.
\end{cor}
\begin{proof}
Since $Qf$ is a $d$-monomorphism in $\scrT\lquot\scrC$, the image of the above bicartesian diagram in $\scrT\lquot\scrC$ is also cartesian by Proposition~\ref{prop:ker_in_ideal_quot} and Proposition~\ref{prop:d_mono_chara_main}. On the other hand, since $\iota$ is a left $\scrC$-approximation of $X$ in $\scrT$, the image of the above bicartesian diagram in $\scrT\lquot\scrC$ is also cocartesian by the dual of Proposition~\ref{prop:ker_in_ideal_quot}. Thus, the image of the above bicartesian diagram in $\scrT\lquot\scrC$ is also bicartesian.
\end{proof}

\begin{cor}
\label{cor:exact_to_exact}
Consider the following exact $3$-term complex in $\scrT$:
$$X\xrightarrow{f} Y\xrightarrow{g} Z \quad (h\colon X\to Z)$$
together with the associated triangle in $\calT$:
$$X\xrightarrow{f} Y\xrightarrow{g} Z \xrightarrow{\delta} X[1]$$
If $\delta[-1]$ factors through an object in $\calC_0^{d-1}$ and $\delta$ factors through an object in $\calC_{-d+1}^0$, then the image of the above $3$-term complex in $\scrT\lquot\scrC$ remains exact. 
\end{cor}
\begin{proof}
It immediately follows from Proposition~\ref{prop:d_mono_chara_main}, Proposition~\ref{prop:d_epi_chara_main} and Proposition~\ref{prop:ker_in_ideal_quot}.
\end{proof}

\begin{theorem}
\label{thm:preabelian}
Let $\calT$ be a stable DG-category and $\calC$ a $(d+1)$-cluster tilting DG-subcategory of $\calT$. Then the DG-quotient $\scrT\lquot\scrC$ is an abelian $d$-truncated DG-category.
\end{theorem}
\begin{proof}
We already know that $\scrT\lquot\scrC$ is additive and $d$-truncated by Proposition~\ref{prop:ideal_quot_d_trun}. We also know that it admits kernels and cokernels by Corollary~\ref{cor:ideal_quot_fin_comp} and its dual. 

Thus we need to show that condition (2) in Definition~\ref{def:abelian_d_DG} holds. Condition (3) in Definition~\ref{def:abelian_d_DG} can be shown by the dual argument. Take a right exact $3$-term complex in $\scrT\lquot\scrC$:
\begin{equation}
\label{eq:right_exact_3-term}
X'\xrightarrow{f'} Y' \xrightarrow{g'} Z' \quad (h'\colon X'\to Z')
\end{equation}
where $f'$ is a $d$-monomorphism in $\scrT\lquot\scrC$. Since the functor 
$$H^0(Q)\colon \calT\to H^0(\scrT\lquot\scrC)$$
is surjective on objects and degree $0$ morphisms, we can take a closed morphism $f\colon X\to Y$ of degree $0$ in $\scrT$ such that $Qf = f'$ in $\scrT\lquot\scrC$. Then we can take a bicartesian diagram in $\scrT$ as follows:
$$\begin{tikzcd}
X & C \\ Y & Z
\Ar{1-1}{1-2}{"\iota"}
\Ar{2-1}{2-2}{"g"'}
\Ar{1-1}{2-1}{"f"'}
\Ar{1-2}{2-2}{"\pi"}
\Ar{1-1}{2-2}{"h", red}
\end{tikzcd}$$
where $C\in\scrC$ and $\iota\colon X\to C$ is a left $\scrC$-approximation of $X$ in $\scrT$. Since $Qf=f'$ is a $d$-monomorphism in $\scrT\lquot\scrC$, we have the following exact $3$-term complex in $\scrT\lquot\scrC$ by Corollary~\ref{cor_exact_in_ideal}:
\begin{equation}
\label{eq:exact_in_ideal}
QX\xrightarrow{Qf} QY \xrightarrow{Qg} QZ \quad (Qh-Q(\pi)\circ h_{QC}\circ Q(\iota)\colon QX\to QZ)
\end{equation}
where $h_{QC}\colon QC\to QC$ is a degree $-1$ morphism such that $dh_{QC}=\id_{QC}$. Since $Qf=f'$, homotopy $3$-term complexes \eqref{eq:right_exact_3-term} and \eqref{eq:exact_in_ideal} are isomorphic in $\cpx^3(\scrT\lquot\scrC)$ by the dual of Lemma~\ref{lem:unique_ker}. Thus, the homotopy $3$-term complex \eqref{eq:right_exact_3-term} is also left exact in $\scrT\lquot\scrC$. This shows that condition (2) in Definition~\ref{def:abelian_d_DG} holds.
\end{proof}

% \textcolor{blue}{We now turn to the study of projective dimensions of obejects in $\scrT\lquot\scrC$.
% The following clarifies a useful discription of projectives in $\scrT\lquot\scrC$ in terms of $\scrT$.}

\begin{proposition}
\label{prop:enough_proj}
The abelian $d$-truncated DG-category $\scrT\lquot\scrC$ has enough projective objects, and the full DG-subcategory of projective objects in $\scrT\lquot\scrC$ corresponds to the additive closure of the essential image $Q(\scrC[-d])$.  
\end{proposition}
\begin{proof}
First, we show that for any object $X'$ in $\scrT\lquot\scrC$, there exists an object $C[-d]\in\scrC[-d]$ and a $d$-epimorphism $Q(C[-d])\to X'$ in $\scrT\lquot\scrC$. We may assume that $X'=QX$ for some object $X\in\scrT$. By Fact~\ref{thm:splicing}, we have an exact $3$-term complex in $\scrT$:
$$C[-d]\xrightarrow{} X \xrightarrow{} Y \quad (h\colon C[-d]\to Y)$$
where $C\in\scrC$ and $Y$ is in $\scrC_{-d+1}^0$. By Proposition~\ref{prop:d_epi_chara_main}, the morphism $Q(C[-d])\to QX$ is a $d$-epimorphism in $\scrT\lquot\scrC$. 

Next, we show that for any object $C[-d]\in\scrC[-d]$, the object $Q(C[-d])$ is projective in $\scrT\lquot\scrC$. Take a $d$-epimorphism $f'\colon X'\to Y'$ in $\scrT\lquot\scrC$ and a morphism $g'\colon Q(C[-d])\to Y'$ in $\scrT\lquot\scrC$. We may assume that $f'=Qf$ for some closed morphism $f\colon X\to Y$ of degree $0$ in $\scrT$ and $g'=Qg$ for some closed morphism $g\colon C[-d]\to Y$ of degree $0$ in $\scrT$. Take a triangle in $\calT$:
$$X\xrightarrow{f} Y\xrightarrow{i} \Cone(f)\to X[1]$$
Since $Qf$ is a $d$-epimorphism in $\scrT\lquot\scrC$, the morphism $i\colon Y\to \Cone(f)$ factors through an object in $\scrC_{-d+1}^0$ by Proposition~\ref{prop:d_epi_chara_main}. Consider the following exact sequence in $\Mod \calC$:
$$\calT(*,X)|_{\calC[-d]} \xrightarrow{(f\circ-)|_{\calC[-d]}} \calT(*,Y)|_{\calC[-d]} \xrightarrow{(i\circ -)|_{\calC[-d]}} \calT(*,\Cone(f))|_{\calC[-d]}.$$
Since the morphism $i\colon Y\to \Cone(f)$ factors through an object in $\scrC[-d+1]*\cdots*\scrC$, we have $(i\circ -)|_{\calC[-d]}=0$ in $\Mod \calC$. Thus, the morphism $(f\circ-)|_{\calC[-d]}\colon \calT(*,X)|_{\calC[-d]} \to \calT(*,Y)|_{\calC[-d]}$ is an epimorphism in $\Mod (\calC[-d])$. This implies that the morphism $g\colon C[-d]\to Y$ factors through $f\colon X\to Y$ in $\calT$. Thus, the morphism $g'\colon Q(C[-d])\to Y'$ factors through $f'\colon X'\to Y'$ in $\scrT\lquot\scrC$. This shows that $Q(C[-d])$ is projective in $\scrT\lquot\scrC$.

Finally, we show that any projective object in $\scrT\lquot\scrC$ is isomorphic to a direct summand of an object in $Q(\scrC[-d])$. Take a projective object $P$ in $\scrT\lquot\scrC$. Then there exists a $d$-epimorphism $f'\colon Q(C[-d])\to P$ in $\scrT\lquot\scrC$ for some object $C\in\scrC$. Since $P$ is projective, the morphism $f'\colon Q(C[-d])\to P$ splits in $\scrT\lquot\scrC$. Thus, $P$ is isomorphic to a direct summand of $Q(C[-d])$.
\end{proof}

By the dual argument, we also have the following result:

\begin{proposition}
\label{prop:enough_inj}
The abelian $d$-truncated DG-category $\scrT\lquot\scrC$ has enough injective objects, and the full sub DG-category of injective objects in $\scrT\lquot\scrC$ corresponds to the additive closure of the essential image $Q(\scrC[d])$.
\end{proposition}

\begin{lemma}
\label{lem:proj_in_ideal_quot_IC}
Assume that the homotopy category $\calT$ is idempotent complete. Then the full sub DG-category of projective objects in $\scrT\lquot\scrC$ coincides with $Q(\scrC[-d])$. Dually, the full sub DG-category of injective objects in $\scrT\lquot\scrC$ coincides with $Q(\scrC[d])$.
\end{lemma}
\begin{proof}
We only show the assertion for projective objects. The assertion for injective objects can be shown by the dual argument. By Proposition~\ref{prop:enough_proj}, the full sub DG-category of projective objects in $\scrT\lquot\scrC$ is given by the additive closure of the essential image $Q(\scrC[-d])$. Thus, it is enough to show that the essential image $Q(\scrC[-d])$ is closed under direct summands in $\scrT\lquot\scrC$. First, we show that the restricted quotient DG-functor $H^0(Q)\colon \calC[-d]\to \calT/[\scrC]$ is fully faithful. Since this functor is already known to be full, it remains to show that it is faithful. Let $f\colon C[-d]\to C'[-d]$ be a morphism in the ideal $[\scrC]$. Then there exists a factorization $C[-d]\to C''\to C'[-d]$ of $f$ in $\calT$ with $C''\in\scrC$. In this situation, the morphism $C[-d]\to C''$ is zero since $\calT(C[-d],\calC)=0$. Thus, the morphism $f\colon C[-d]\to C'[-d]$ is zero in $\calT/[\scrC]$. This shows that the functor $H^0(Q)\colon \calC[-d]\to \calT/[\scrC]$ is fully faithful. Since $\calC[-d]$ is closed under direct summands in $\calT$ and $\calT$ is idempotent complete, the essential image $Q(\calC[-d])$ is also idempotent complete. Thus, $Q(\calC[-d])$ is closed under direct summands in $\scrT\lquot\scrC$.
\end{proof}

The following proposition gives a useful description of the full sub DG-category $Q(\scrC[-d])$. This implies that the full sub DG-category $Q(\scrC[-d])$ is quasi-equivalent to the truncated full sub DG-category $\tau^{>-d}(\scrC[-d])$ in $\tau^{>-d}\scrT$.

\begin{proposition}
\label{prop:end}
The quotient DG-functor $Q\colon \scrT\to \scrT\lquot\scrC$ induces an isomorphism 
$$H^i\bigl(\scrT(X,Y)\bigr) \cong H^i\bigl((\scrT\lquot\scrC)(QX,QY)\bigr)$$
for all $-d<i\leq 0$ and $X\in\scrC[-d]$, $Y\in\scrT$.
\end{proposition}
\begin{proof}
Let $X$ be an object in $\scrC[-d]$ and $Y$ be an object in $\scrT$. Then we have the following triangle in $\calD(k)$:
$$\scrT(-,Y)|_{\scrC}\lox_{\scrC}\scrT(X,-)|_{\scrC\op}\xrightarrow{m^{\scrC}_{Y}} \scrT(X,Y) \xrightarrow{Q_{Y}} (\scrT\lquot\scrC)(QX,QY)\dashrightarrow.$$
Since $X\in\scrC[-d]$, we have 
$$H^i(\scrT(X,-)|_{\scrC\op}) \cong \calT(X,\calC[i])= 0 $$
for any $-d<i\leq 0$. Moreover, since $\scrT$ is connective, we also have $H^i(\scrT(X,-)|_{\scrC\op}) = 0$ for any $i>0$. Since $X\in\scrC[-d]$, $\scrC$ is a $(d+1)$-cluster tilting DG-subcategory of $\scrT$, and $\scrT$ is a connective DG-category. Thus, we have an isomorphism:
$$\scrT(-,Y)|_{\scrC}\lox_{\scrC}\scrT(X,-)|_{\scrC\op}\cong \scrT(-,Y)|_{\scrC}\lox_{\scrC}\tau^{\leq -d}\bigl(\scrT(X,-)|_{\scrC\op}\bigr).$$
Thus, we have $H^i\bigl(\scrT(-,Y)|_{\scrC}\lox_{\scrC}\scrT(X,-)|_{\scrC\op}\bigr)=0$ for any $i>-d$. This implies that the morphism $Q_{X}\colon \scrT(X,Y) \to (\scrT\lquot\scrC)(QX,QY)$ induces the following isomorphism for any $-d<i\leq 0$:
$$H^i\bigl(\scrT(X,Y)\bigr) \cong H^i\bigl((\scrT\lquot\scrC)(QX,QY)\bigr).$$
This shows the desired isomorphism.
\end{proof}

Next, we show that the DG-quotient $\scrT\lquot\scrC$ is quasi-equivalent to the DG-category of finitely presented $d$-extended modules over the truncated full sub DG-category $\tau^{>-d}\scrC$. We first give a precise definition of the DG-category of finitely presented $d$-extended modules over a connective DG-category.

\begin{definition}
\label{def:demfpdg}
Let $\scrB$ be a connective DG-category. Define the $d$-extended module DG-category $\demdg{\scrB}$ as the full DG-subcategory of $\scrD(\scrB)$ consisting of objects $M$ such that $H^i(M)=0$ for any $i \leq -d$ and $i> 0$. An object $M\in \demdg{\scrB}$ is called a finitely presented $d$-extended module if there exist the following conflations in the extriangulated category $H^0(\demdg{\scrB})$:
$$M^{-(i+1)}\to P^{-i}\to M^{-i} \dashrightarrow $$
for each $0\leq i \leq d$ where $P^0, P^{-1}, \ldots, P^{-d}$ are in $\add \scrB$ and $M^0=M$. We denote by $\demfpdg{\scrB}$ the full DG-subcategory of $\demdg{\scrB}$ consisting of finitely presented objects.
\end{definition}

\begin{remark}
If $k$ is a field and $\scrB$ is a locally finite connective DG-category, then the DG-category $\demfpdg{\scrB}$ corresponds to the full DG-subcategory of $\demdg{\scrB}$ consisting of objects $M$ such that $H^i(M)$ is finite-dimensional. In this case, the DG-category $\demfpdg{\scrB}$ is also denoted by $\demfddg{\scrB}$. 
\end{remark}

\begin{remark}
\label{rem:demfpdg}
Let $\scrB$ be a connective DG-category. Then the DG-category $\demdg{\scrB}$ is naturally quasi-equivalent to the $d$-extended module DG-category $\demdg{\tau^{>-d}\scrB}$ over the truncated full sub DG-category $\tau^{>-d}\scrB$. Thus, we can assume that $\scrB$ is $d$-truncated when we consider the DG-category $\demdg{\scrB}$. 
\end{remark}

% \begin{remark}
% If $k$ is a field and $\scrB$ is a locally finite connective DG-category, then the DG-category $\demfpdg{\scrB}$ corresponds to the full DG-subcategory of $\demdg{\scrB}$ consisting of objects $M$ such that $H^i(M)$ is finite-dimensional.
% \end{remark}

\begin{lemma}
\label{lem:restriction}
Let $\scrB$ and $\scrB'$ be locally $K$-projective connective DG-categories, and let $F\colon \scrB\to \scrB'$ be a DG-functor. Then the restriction DG-functor $(-)|_{\scrB}$ can be restricted to a DG-functor $(-)|_{\scrB}\colon \demdg{\scrB'}\to \demdg{\scrB}$.
\end{lemma}
\begin{proof}
It is clear that the restriction DG-functor $(-)|_{\scrB}\colon \scrD(\scrB')\to \scrD(\scrB)$ preserves the $t$-structure. Thus, it induces a DG-functor $(-)|_{\scrB}\colon \demdg{\scrB'}\to \demdg{\scrB}$. 
\end{proof}

\begin{theorem}
\label{thm:main}
Let $\scrT$ be an essentially small, locally $K$-projective and stable DG-category and $\calC$ a $(d+1)$-cluster tilting DG-subcategory of $\scrT$. Then the DG-functor 
$$F_\scrC:=\tau^{>-d}\bigl(\scrT(*[-d],-)|_{\scrC}\bigr)\colon \scrT \to \demfpdg{\tau^{>-d}\scrC}$$
induces an isomorphism $\scrT\lquot\scrC \cong \demfpdg{\tau^{>-d}\scrC}$ in $\hqe$ which makes the following diagram commutative in $\hqe$:
$$\begin{tikzcd}
\scrC\ar[r,"\mathrm{incl}"]\ar[rrd,"0"', bend right =15] & \scrT \ar[r,"Q"] \ar[rd, "{F_\scrC}"'] & \scrT\lquot\calC \ar[d, "\sim"] \\
&& \demfpdg{\tau^{>-d}\scrC}
\end{tikzcd}$$
\end{theorem}
\begin{proof}
Note that, by Corollary~\ref{cor:morita}, we already have a quasi-equivalence DG-functor 
$$G:=(\scrT\lquot\scrC)(-,*)|_{Q(\scrC[-d])} \colon \scrT\lquot\scrC \to \demfpdg{Q(\tau^{>-d}(\scrC[-d]))}$$
since the full DG-subcategory $Q(\scrC[-d])$ is a progenerating full DG-subcategory of $\scrT\lquot\scrC$ by Proposition~\ref{prop:enough_proj}.
Then we have the following diagram in $\hqe$: 
$$\begin{tikzcd}[column sep=50]
\scrT & & \scrT\lquot\scrC \\
\demfpdg{\tau^{>-d}\scrC} & \demfpdg{\tau^{>-d}(\scrC[-d])} & \demfpdg{Q(\tau^{>-d}(\scrC[-d]))}
\Ar{1-1}{1-3}{"Q"}
\Ar{1-1}{2-1}{"F_\scrC"'}
\Ar{1-3}{2-3}{"G"}
\Ar{2-3}{2-2}{"{(-)|_{\tau^{>-d}(\scrC[-d])}}"}
\Ar{2-2}{2-1}{"(-)|_{\tau^{>-d}\scrC}"}
\end{tikzcd}$$
where the bottom horizontal morphism is an isomorphism in $\hqe$ which is induced by the following quasi-equivalence DG-functors respectively:
$$\tau^{>-d}Q\colon \tau^{>-d}(\scrC[-d]) \to Q(\tau^{>-d}(\scrC[-d]))\quad \text{and} \quad \tau^{>-d}[-d]\colon \tau^{>-d}\scrC \to \tau^{>-d}(\scrC[-d]).$$
We need to show that the above diagram is commutative in $\hqe$. 
It follows from the following isomorphism:
\begin{align*}
\bigl((-)|_{\tau^{>-d}\scrC} \circ (-)|_{\tau^{>-d}(\scrC[-d])} \circ G \circ Q\bigr) & = \bigl((\scrT\lquot\scrC)(Q(*), Q(-))|_{\tau^{>-d}(\scrC[-d])}\bigr)|_{\tau^{>-d}\scrC} \\
& \cong \tau^{>-d}\bigl(\scrT(*[-d],-)|_{\scrC}\bigr) \\
& = F_\scrC.
\end{align*}
The isomorphism in the second line is induced by $Q$ and Proposition~\ref{prop:end}. 
\end{proof}

% \begin{remark}
% \label{rem:main}
% In Theorem~\ref{thm:main}, we can replace $\tau^{>-d}(\scrC[-d])$ with $\tau^{>-d}\scrC$ since $\scrC$ and $\scrC[-d]$ are isomorphic in $\hqe$ by the shift morphism $[d]\colon \scrC\to \scrC[-d]$. 
% \end{remark}

The following theorem is the $(d+1)$-cluster tilting-object case of Theorem~\ref{thm:main}.

\begin{cor}
\label{cor:main_one_obj}
Let $\scrT$ be an essentially small, locally $K$-projective stable DG-category  and $M$ a $(d+1)$-cluster tilting object in $\scrT$. Let $\Lambda:=\tau^{>-d}\End_{\scrT}(M)$ be a $d$-truncated endomorphism DG-algebra of $M$. Then the DG-functor
$$F_M:=\tau^{>-d}\bigl(\scrT(M[-d],*)\bigr)\colon \scrT \to \demfpdg{\Lambda}$$
induces an isomorphism $\scrT\lquot\add(M) \cong \demfpdg{\Lambda}$ in $\hqe$ which makes the following diagram commutative in $\hqe$:
$$\begin{tikzcd}\add(M)\ar[r,"\mathrm{incl}"]\ar[rrd,"0"', bend right =15] & \scrT \ar[r,"Q"] \ar[rd, "{F_M}"'] & \scrT\lquot\add(M) \ar[d, "\sim"] \\
&& \demfpdg{\Lambda}\end{tikzcd}$$
\end{cor}
\begin{proof}
Since $M$ is a $(d+1)$-cluster tilting object in $\scrT$, the full sub DG-category $\add(M)$ is a $(d+1)$-cluster tilting DG-subcategory of $\scrT$. Thus, the assertion follows from Theorem~\ref{thm:main}.
\end{proof}

% We can expect the $d$-truncated endomorphism DG-algebra $\Lambda$ of $M$ has a good properties like a Gorensteinness. 

% \begin{definition}
% \label{def:d-shifted_gorenstein}
% Let $\Lambda$ be a proper $d$-truncated connective DG-algebra. We say that $\Lambda$ has \emph{right $d$-shifted Gorenstein dimension at most $n$} if $\Lambda\in\dem{\Lambda}$ has injective dimension at most $n$ in $\dem{\Lambda}$. If $\Lambda\op$ also has right $d$-shifted Gorenstein dimension at most $n$, then we say that $\Lambda$ has \emph{$d$-shifted Gorenstein dimension at most $n$}.
% \end{definition}

% \begin{lemma}
% \label{lem:symmetry_of_d_gor_dim}
% Let $\Lambda$ be a proper $d$-truncated connective DG-algebra. Assume that $\Lambda$ has finite injective dimension in $\dem{\Lambda}$ and $\Lambda\op$ also has finite injective dimension in $\dem{\Lambda\op}$. Then $\Lambda$ has $d$-shifted Gorenstein dimension at most $n$ if and only if $\Lambda\op$ has $d$-shifted Gorenstein dimension at most $n$. 
% \end{lemma}
% \begin{proof}
% \textcolor{blue}{We need to proof}
% \end{proof}

% \begin{theorem}
% \label{thm:Gor}
% Let $\calT$ be a locally finite stable DG-category and $M$ a $(d+1)$-cluster tilting object in $\calT$. Let $\Lambda:=\tau^{>-d}\End_{\scrT}(M)$ be a $d$-truncated endmorphism DG-algebra of $M$. Then $\Lambda$ has $d$-shifted Gorenstein dimension at most $d$.
% \end{theorem}
% \begin{proof}
% It follows from Proposition~\ref{prop:Gorensteinness_Abel} and Lemma~\ref{lem:symmetry_of_d_gor_dim}.
% \end{proof}

\begin{proposition}
\label{prop:self_inj}
Assume that the homotopy category $\calT$ of $\scrT$ is idempotent complete. Then the abelian $d$-truncated DG-category $\scrT\lquot\scrC$ is Frobenius (that is, it has enough projectives and injectives, and projectives coincide with injectives) if and only if $\scrC[d]=\scrC[-d]$ in $\calT$.
\end{proposition}
\begin{proof}
It immediately follows from Propositions~\ref{prop:enough_proj}, \ref{prop:enough_inj} and Lemma~\ref{lem:proj_in_ideal_quot_IC}. 
\end{proof}

This proposition gives a criterion for $\Lambda$ in Corollary~\ref{cor:main_one_obj} to be $d$-self-injective. We give a definition of $d$-self-injective DG-algebras.

\begin{definition}
\cite[Definition~2.2]{MR4133519}
\label{def:d_self_inj}
Assume that $k$ is a field. Let $\Lambda$ be a proper $d$-truncated connective DG-algebra. We say that $\Lambda$ is \emph{$d$-self-injective} if $\Lambda\cong D(\Lambda)[d]$ in $\calD(\Lambda)$ where $D(\Lambda):=\rhom{k}{\Lambda, k}$ is the $k$-dual of $\Lambda$. Equivalently, $\Lambda$ is injective in $\dem{\Lambda}$.
\end{definition}

\begin{cor}
\label{cor:self_inj}
Assume that $k$ is a field. Let $\calT$ be a locally finite stable DG-category with idempotent complete homotopy category $\calT$ and $M$ a $(d+1)$-cluster tilting object in $\calT$. Let $\Lambda:=\tau^{>-d}\End_{\scrT}(M)$ be a $d$-truncated endomorphism DG-algebra of $M$. Then $\Lambda$ is $d$-self-injective if and only if $\add(M[d])=\add(M[-d])$ in $\calT$.
\end{cor}
\begin{proof}
It immediately follows from Proposition~\ref{prop:self_inj}. 
\end{proof}

\subsection{From Amiot-Guo's generalized cluster categories}

In this section, we consider the case where $\calT$ is a generalized cluster category introduced by Amiot and Guo. We assume that $k$ is a field.

Let $\Pi$ be a connective DG-algebra over $k$, and let $\Pi^e$ be the enveloping DG-algebra $\Pi\otimes_k \Pi\op$. 

\begin{definition}
\label{def:calabi_yau_dga}
If $\Pi$ satisfies the following conditions, then we say that $\Pi$ is a \emph{$(d+2)$-Calabi-Yau DG-algebra}:
\begin{itemize}
	\item $\Pi$ is homologically smooth, i.e., $\Pi$ is perfect as a DG-module over $\Pi^e$.
	\item $\Pi$ is bimodule $(d+2)$-Calabi-Yau, i.e., there exists an isomorphism in $\calD(\Pi^e)$
	$$\rhom{\Pi^e}{\Pi, \Pi^e} \cong \Pi[-d-2].$$
	\item $H^0(\Pi)$ is finite-dimensional over $k$.
\end{itemize}
\end{definition}

\begin{definition}
\label{def:gen_cluster_cat}
Let $\Pi$ be a $(d+2)$-Calabi-Yau DG-algebra. Then the \emph{generalized $(d+1)$-cluster category} $\calC_{d+1}(\Pi)$ associated with $\Pi$ is defined as the Verdier quotient
$$\calC_{d+1}(\Pi) := \per \Pi/ \pvd \Pi.$$
We also define the \emph{generalized $(d+1)$-cluster DG-category} $\scrC_{d+1}(\Pi)$ associated with $\Pi$ as the DG-quotient
$$\scrC_{d+1}(\Pi) := \tau^{\leq 0}(\per\dg \Pi\lquot \pvd\dg \Pi).$$
where $\per\dg \Pi$ and $\pvd\dg \Pi$ are natural pretriangulated DG-categories which enhance $\per \Pi$ and $\pvd \Pi$, respectively. Of course, we have $H^0(\scrC_{d+1}(\Pi))\cong \calC_{d+1}(\Pi)$ as triangulated categories.
\end{definition}

\begin{remark}
\label{rem:cluster_stable}
Since $\per\dg \Pi\lquot \pvd\dg \Pi$ is pretriangulated, the generalized $(d+1)$-cluster DG-category $\scrC_{d+1}(\Pi)$ is a stable DG-category by Example~\ref{ex:stable_DG}.
\end{remark}

From now on, we fix a $(d+2)$-Calabi-Yau DG-algebra $\Pi$ and denote the image of $\Pi$ in $\scrC_{d+1}(\Pi)$ by $M_\Pi$.

\begin{fact}\cite[Theorem~4.22]{Guo11}
\label{thm:guo's theorem}
The following conditions hold:
\begin{enumerate}
\item The category $\calC_{d+1}(\Pi)$ is a Hom-finite $(d+1)$-Calabi-Yau triangulated category.
\item The full subcategory $\add(M_\Pi)$ is a $(d+1)$-cluster tilting subcategory of $\calC_{d+1}(\Pi)$.
\end{enumerate}
\end{fact}

\begin{remark}
\label{rem:AR}
Fact~\ref{thm:guo's theorem} implies that the generalized $(d+1)$-cluster category $\calC_{d+1}(\Pi)$ has a Serre functor and hence has Auslander-Reiten triangles. In particular, the AR-translation $\tau$ of $\calC_{d+1}(\Pi)$ satisfies $\tau\cong [d]$.
\end{remark}

The following is just a reformulation of Fact~\ref{thm:guo's theorem}.

\begin{cor}
\label{cor:guo's theorem}
The following conditions hold:
\begin{enumerate}
\item The DG-category $\scrC_{d+1}(\Pi)$ is a locally finite stable DG-category.
\item The full sub DG-category $\add(M_\Pi)$ is a $(d+1)$-cluster tilting DG-subcategory of $\scrC_{d+1}(\Pi)$.
\end{enumerate}
\end{cor}

The following result is a direct consequence of \cite[Proposition~2.15]{Guo11}. 

\begin{fact}[{\cite[Proposition~2.15]{Guo11}}]
\label{prop:fundamental_domain}
The DG-functor $\tau^{\leq 0}\pi\colon \tau^{\leq 0}(\per\dg\Pi)\to \scrC_{d+1}(\Pi)$ induces a quasi-equivalence 
$$\tau^{\leq 0}(\add\Pi*\add\Pi[1]*\cdots*\add\Pi[d]) \to \scrC_{d+1}(\Pi).$$
\end{fact}

\begin{remark}
\label{rem:fundamental_domain}
The left-hand side above is called the \emph{fundamental domain} of $\scrC_{d+1}(\Pi)$.
\end{remark}

\begin{cor}\footnote{This was suggested by Junyang Liu.}
\label{cor:fundamental_domain}
The DG-functor $\tau^{\leq 0}\pi\colon \tau^{\leq 0}(\per\dg\Pi)\to \scrC_{d+1}(\Pi)$ induces a quasi-isomorphism of DG-algebras:
$$\tau^{>-d}\Pi\cong\tau^{>-d}\rend{\Pi}{\Pi} \to \tau^{>-d}\End_{\scrC_{d+1}(\Pi)}(M_\Pi).$$
\end{cor}
\begin{proof}
By Fact~\ref{prop:fundamental_domain}, we have the following isomorphisms:
\begin{align*}
	H^i(\rend{\Pi}{\Pi})\cong &H^i(\rhom{\Pi}{\Pi, \Pi})\cong H^0(\rhom{\Pi}{\Pi[-i],\Pi})\\
	\cong &H^0(\scrC_{d+1}(\Pi)(M_\Pi[-i],M_\Pi))\cong H^i(\End_{\scrC_{d+1}(\Pi)}(M_\Pi)).
\end{align*}
for any $-d+1\leq i \leq 0$. Thus, we have the desired quasi-isomorphism of DG-algebras.
\end{proof}

\begin{theorem}
\label{thm:cluster_module_cat}
The DG-functor $F_{M_\Pi}\colon \scrC_{d+1}(\Pi) \to \demfddg{\tau^{>-d}\Pi}$ induces an isomorphism 
$$\scrC_{d+1}(\Pi)\lquot\add(M_\Pi) \cong \demfddg{\tau^{>-d}\Pi}$$
in $\hqe$ which makes the following diagram commutative in $\hqe$:
$$\begin{tikzcd}\add(M_\Pi)\ar[r,"\mathrm{incl}"]\ar[rrd,"0"', bend right =15] & \scrC_{d+1}(\Pi) \ar[r,"Q"] \ar[rd, "{F_{M_\Pi}}"'] & \scrC_{d+1}(\Pi)\lquot\add(M_\Pi) \ar[d, "\sim"] \\&& \demfddg{\tau^{>-d}\Pi}\end{tikzcd}$$
\end{theorem}
\begin{proof}
It follows from Theorem~\ref{thm:main}, Corollary~\ref{cor:fundamental_domain}, and Fact~\ref{thm:guo's theorem}.
\end{proof}

\begin{proposition}
\label{prop:cluster_module_cat}
The $d$-truncated DG-algebra $\tau^{>-d}\Pi$ is a $d$-self-injective DG-algebra if and only if $\add\tau M_\Pi =\add\tau^{-1} M_\Pi$ holds in $\calC_{d+1}(\Pi)$.
\end{proposition}
\begin{proof}
It follows from Corollary~\ref{cor:self_inj} since the AR-translation $\tau$ of $\calC_{d+1}(\Pi)$ satisfies $\tau\cong [d]$ by Remark~\ref{rem:AR}.
\end{proof}

\subsection{Examples}

We give an example of Theorem~\ref{thm:cluster_module_cat} in the case where $\scrT$ is a cluster category associated with type $A$ quiver. We refer to \cite{moc2025b} for Auslander--Reiten quivers of $d$-extended module categories.

\begin{example}
In this example, we consider the case when $d=2$ and $\scrT$ is the $3$-cluster category associated with quiver $A_2$. Thus, we have $\scrT\cong \scrC_3(\Pi)$ where $\Pi$ is the $4$-Calabi-Yau DG-algebra associated with $A_2$. 

Then the Auslander--Reiten quiver of $\scrT$ is as follows:
$$
\begin{tikzcd}[column sep =8]
\cdots& P_3^2 && P_0^2 && P_1^2 && P_2^2 && P_3^2 && P_0^2 \\
P_3^1 && P_0^1 && P_1^1 && P_2^1 && P_3^1 && P_0^1 & \cdots
\Ar{2-1}{1-2}{}
\Ar{1-2}{2-3}{}
\Ar{2-3}{1-4}{}
\Ar{1-4}{2-5}{}
\Ar{2-5}{1-6}{}
\Ar{1-6}{2-7}{}
\Ar{2-7}{1-8}{}
\Ar{1-8}{2-9}{}
\Ar{2-9}{1-10}{}
\Ar{1-10}{2-11}{}
\Ar{2-11}{1-12}{}
\tauAr{2-3}{2-1}{}
\tauAr{2-5}{2-3}{}
\tauAr{2-7}{2-5}{}
\tauAr{2-9}{2-7}{}
\tauAr{2-11}{2-9}{}
\tauAr{1-4}{1-2}{}
\tauAr{1-6}{1-4}{}
\tauAr{1-8}{1-6}{}
\tauAr{1-10}{1-8}{}
\tauAr{1-12}{1-10}{}
\end{tikzcd}
$$
The shifts are given by $P_j^1 \mapsto P_{j+1}^2$ and $P_j^2 \mapsto P_{j+2}^1$ where $j\in \mathbb{Z}/4\mathbb{Z}$. We have $\tau=[2]$. Define $M:=P_0^1\oplus P_2^1$. Then $M$ is a $3$-cluster tilting object in $\scrT$ and it satisfies $\add\tau M=\add\tau^{-1}M$ in $\scrT$.
Then the ideal quotient $\scrT/\add M$ has the following Auslander--Reiten quiver: 
$$
\begin{tikzcd}[column sep =8]
\cdots& P_3^2 && P_0^2 && P_1^2 && P_2^2 && P_3^2 && P_0^2 \\
P_3^1 && {} && P_1^1 && {} && P_3^1 && {} & \cdots
\Ar{2-1}{1-2}{}
% \Ar{1-2}{2-3}{}
% \Ar{2-3}{1-4}{}
\Ar{1-4}{2-5}{}
\Ar{2-5}{1-6}{}
% \Ar{1-6}{2-7}{}
% \Ar{2-7}{1-8}{}
\Ar{1-8}{2-9}{}
\Ar{2-9}{1-10}{}
% \Ar{1-10}{2-11}{}
% \Ar{2-11}{1-12}{}
% \tauAr{2-3}{2-1}{}
% \tauAr{2-5}{2-3}{}
% \tauAr{2-7}{2-5}{}
% \tauAr{2-9}{2-7}{}
% \tauAr{2-11}{2-9}{}
\tauAr{1-4}{1-2}{}
\tauAr{1-6}{1-4}{}
\tauAr{1-8}{1-6}{}
\tauAr{1-10}{1-8}{}
\tauAr{1-12}{1-10}{}
\end{tikzcd}
$$
On the other hand, we can check that the endomorphism DG-algebra $\tau_{>-2}\left(\End_{\scrC_3(\Pi)}(M)\right)$ is given by the following DG-quiver with relations:
$$k[\begin{tikzcd}
1 & 2 
\Ar{1-1}{1-2}{"\alpha", red, shift left =0.5ex}
\Ar{1-2}{1-1}{"\beta", red, shift left =0.5ex}
\end{tikzcd}]/(\beta\alpha, \alpha\beta)$$
with zero differential. Then its $2$-extended module category has the following Auslander--Reiten quiver:

$$
\begin{tikzcd}[column sep =8]
\cdots& 
{\color{black}
\begin{xsmallmatrix}
2
\end{xsmallmatrix}
}
 && 
{\color{red}
\begin{xsmallmatrix}
2
\end{xsmallmatrix}
}
 && 
{\color{black}
\begin{xsmallmatrix}
1
\end{xsmallmatrix}
}
 && 
{\color{red}
\begin{xsmallmatrix}
1
\end{xsmallmatrix}
}
 && 
{\color{black}
\begin{xsmallmatrix}
2
\end{xsmallmatrix}
}
 && 
{\color{red}
\begin{xsmallmatrix}
2
\end{xsmallmatrix}
}
 \\
{
\begin{xsmallmatrix}
\color{black}2\\\color{red}1
\end{xsmallmatrix}
}
 && {} && 
{
\begin{xsmallmatrix}
\color{black}1\\\color{red}2
\end{xsmallmatrix}
}
 && {} && 
{
\begin{xsmallmatrix}
\color{black}2\\\color{red}1
\end{xsmallmatrix}
}
 && {} & \cdots
\Ar{2-1}{1-2}{}
% \Ar{1-2}{2-3}{}
% \Ar{2-3}{1-4}{}
\Ar{1-4}{2-5}{}
\Ar{2-5}{1-6}{}
% \Ar{1-6}{2-7}{}
% \Ar{2-7}{1-8}{}
\Ar{1-8}{2-9}{}
\Ar{2-9}{1-10}{}
% \Ar{1-10}{2-11}{}
% \Ar{2-11}{1-12}{}
% \tauAr{2-3}{2-1}{}
% \tauAr{2-5}{2-3}{}
% \tauAr{2-7}{2-5}{}
% \tauAr{2-9}{2-7}{}
% \tauAr{2-11}{2-9}{}
\tauAr{1-4}{1-2}{}
\tauAr{1-6}{1-4}{}
\tauAr{1-8}{1-6}{}
\tauAr{1-10}{1-8}{}
\tauAr{1-12}{1-10}{}
\end{tikzcd}
$$
\end{example}

\begin{example}
In this example, we consider the case when $d=2$ and $\scrT$ is the $3$-cluster category associated with quiver $A_3$. Thus, we have $\scrT\cong \scrC_3(\Pi)$ where $\Pi$ is the $4$-Calabi-Yau DG-algebra associated with $A_3$.

Then the Auslander--Reiten quiver of $\scrT$ is as follows:
$$
\begin{tikzcd}[column sep =8]
\cdots && P_4^3 && P_0^3 && P_1^3 && P_2^3 && P_3^3 && P_4^3 && P_0^3 \\
& P_4^2 && P_0^2 && P_1^2 && P_2^2 && P_3^2 && P_4^2 && P_0^2 \\
P_4^1 && P_0^1 && P_1^1 && P_2^1 && P_3^1 && P_4^1 && P_0^1 & \cdots
\Ar{3-1}{2-2}{}
\Ar{2-2}{3-3}{}
\Ar{3-3}{2-4}{}
\Ar{2-4}{3-5}{}
\Ar{3-5}{2-6}{}
\Ar{2-6}{3-7}{}
\Ar{3-7}{2-8}{}
\Ar{2-8}{3-9}{}
\Ar{3-9}{2-10}{}
\Ar{2-10}{3-11}{}
\Ar{3-11}{2-12}{}
\Ar{2-12}{3-13}{}
\Ar{3-13}{2-14}{}
\Ar{2-2}{1-3}{}
\Ar{1-3}{2-4}{}
\Ar{2-4}{1-5}{}
\Ar{1-5}{2-6}{}
\Ar{2-6}{1-7}{}
\Ar{1-7}{2-8}{}
\Ar{2-8}{1-9}{}
\Ar{1-9}{2-10}{}
\Ar{2-10}{1-11}{}
\Ar{1-11}{2-12}{}
\Ar{2-12}{1-13}{}
\Ar{1-13}{2-14}{}
\Ar{2-14}{1-15}{}
\tauAr{3-3}{3-1}{}
\tauAr{3-5}{3-3}{}
\tauAr{3-7}{3-5}{}
\tauAr{3-9}{3-7}{}
\tauAr{3-11}{3-9}{}
\tauAr{3-13}{3-11}{}
\tauAr{2-4}{2-2}{}
\tauAr{2-6}{2-4}{}
\tauAr{2-8}{2-6}{}
\tauAr{2-10}{2-8}{}
\tauAr{2-12}{2-10}{}
\tauAr{2-14}{2-12}{}
\tauAr{1-5}{1-3}{}
\tauAr{1-7}{1-5}{}
\tauAr{1-9}{1-7}{}
\tauAr{1-11}{1-9}{}
\tauAr{1-13}{1-11}{}
\tauAr{1-15}{1-13}{}
\end{tikzcd}
$$

Shifts are given by $P^1_i \mapsto P^3_{i+1}$, $P^2_i \mapsto P^2_{i+2}$ and $P^3_i \mapsto P^1_{i+3}$ where $i\in \mathbb{Z}/5\mathbb{Z}$. We have $\tau=[2]$ and $\tau^{-1}=[-2]$. Define $M:=P_0^1\oplus P_0^2\oplus P_3^1$. Then $M$ is a $3$-cluster tilting object in $\scrT$ and it does not satisfy $\add\tau M=\add\tau^{-1}M$ in $\scrT$.

Then the ideal quotient $\scrT/\add M$ has the following Auslander--Reiten quiver:
$$
\begin{tikzcd}[column sep =8]
\cdots && P_4^3 && P_0^3 && P_1^3 && P_2^3 && P_3^3 && P_4^3 && P_0^3 \\
& P_4^2 && {} && P_1^2 && P_2^2 && P_3^2 && P_4^2 && {} \\
P_4^1 && {} && P_1^1 && P_2^1 && {} && P_4^1 && {} & \cdots
\Ar{3-1}{2-2}{}
% \Ar{2-2}{3-3}{}
% \Ar{3-3}{2-4}{}
% \Ar{2-4}{3-5}{}
\Ar{3-5}{2-6}{}
\Ar{2-6}{3-7}{}
\Ar{3-7}{2-8}{}
% \Ar{2-8}{3-9}{}
% \Ar{3-9}{2-10}{}
\Ar{2-10}{3-11}{}
\Ar{3-11}{2-12}{}
% \Ar{2-12}{3-13}{}
% \Ar{3-13}{2-14}{}
\Ar{2-2}{1-3}{}
% \Ar{1-3}{2-4}{}
% \Ar{2-4}{1-5}{}
\Ar{1-5}{2-6}{}
\Ar{2-6}{1-7}{}
\Ar{1-7}{2-8}{}
\Ar{2-8}{1-9}{}
\Ar{1-9}{2-10}{}
\Ar{2-10}{1-11}{}
\Ar{1-11}{2-12}{}
\Ar{2-12}{1-13}{}
% \Ar{1-13}{2-14}{}
% \Ar{2-14}{1-15}{}
\tauAr{1-5}{1-3}{}
\tauAr{1-7}{1-5}{}
\tauAr{1-9}{1-7}{}
\tauAr{1-11}{1-9}{}
\tauAr{1-13}{1-11}{}
\tauAr{1-15}{1-13}{}
% \tauAr{2-4}{2-2}{}
% \tauAr{2-6}{2-4}{}
\tauAr{2-8}{2-6}{}
\tauAr{2-10}{2-8}{}
\tauAr{2-12}{2-10}{}
% \tauAr{2-14}{2-12}{}
% \tauAr{1-3}{1-1}{}
% \tauAr{1-5}{1-3}{}
\tauAr{3-7}{3-5}{}
% \tauAr{3-9}{3-7}{}
% \tauAr{3-11}{3-9}{}
% \tauAr{3-13}{3-11}{}
\end{tikzcd}
$$

On the other hand, we can check that the endomorphism DG-algebra $\tau_{>-2}\left(\End_{\scrC_3(\Pi)}(M)\right)$ is given by the following DG-quiver with relations:
$$k[\begin{tikzcd}
1 & 2 & 3
\Ar{1-1}{1-2}{"\alpha"}
\Ar{1-2}{1-3}{"\beta", red, shift left =0.5ex}
\Ar{1-3}{1-2}{"\gamma", red, shift left =0.5ex}
\end{tikzcd}]/(\beta\gamma, \gamma\beta)$$
with zero differential. Then its $2$-extended module category has the following Auslander--Reiten quiver:
$$
\begin{tikzcd}[column sep =8]
\cdots && 
{\color{black}
\begin{xsmallmatrix}
3
\end{xsmallmatrix}}
&& 
{\color{red}
\begin{xsmallmatrix}
3
\end{xsmallmatrix}}
&& 
{
\begin{xsmallmatrix}
2\\1
\end{xsmallmatrix}}
&& 
{\color{red}
\begin{xsmallmatrix}
1
\end{xsmallmatrix}}
&& 
{\color{red}
\begin{xsmallmatrix}
2
\end{xsmallmatrix}}
&& 
{\color{black}
\begin{xsmallmatrix}
3
\end{xsmallmatrix}}
&& 
{\color{red}
\begin{xsmallmatrix}
3
\end{xsmallmatrix}}
\\
& 
{
\begin{xsmallmatrix}
\color{black}3\\\color{red}2
\end{xsmallmatrix}}
&& {} && 
{
\begin{xsmallmatrix}
&2&\\
1&&\color{red}3
\end{xsmallmatrix}}
&& 
{
\begin{xsmallmatrix}
2
\end{xsmallmatrix}}
&& 
{
\begin{xsmallmatrix}
\color{red}2\\1
\end{xsmallmatrix}}
&& 
{
\begin{xsmallmatrix}
\color{black}3\\\color{red}2
\end{xsmallmatrix}}
&& {}
\\
{
\begin{xsmallmatrix}
\color{black}3\\\color{red}1
\end{xsmallmatrix}}
&& {} && 
{
\begin{xsmallmatrix}
1
\end{xsmallmatrix}}
&& 
{
\begin{xsmallmatrix}
\color{black}2\\\color{red}3
\end{xsmallmatrix}}
&& {} && 
{
\begin{xsmallmatrix}
\color{black}3\\\color{red}2\\\color{red}1
\end{xsmallmatrix}}
&& {} & \cdots
\Ar{3-1}{2-2}{}
% \Ar{2-2}{3-3}{}
% \Ar{3-3}{2-4}{}
% \Ar{2-4}{3-5}{}
\Ar{3-5}{2-6}{}
\Ar{2-6}{3-7}{}
\Ar{3-7}{2-8}{}
% \Ar{2-8}{3-9}{}
% \Ar{3-9}{2-10}{}
\Ar{2-10}{3-11}{}
\Ar{3-11}{2-12}{}
% \Ar{2-12}{3-13}{}
% \Ar{3-13}{2-14}{}
\Ar{2-2}{1-3}{}
% \Ar{1-3}{2-4}{}
% \Ar{2-4}{1-5}{}
\Ar{1-5}{2-6}{}
\Ar{2-6}{1-7}{}
\Ar{1-7}{2-8}{}
\Ar{2-8}{1-9}{}
\Ar{1-9}{2-10}{}
\Ar{2-10}{1-11}{}
\Ar{1-11}{2-12}{}
\Ar{2-12}{1-13}{}
% \Ar{1-13}{2-14}{}
% \Ar{2-14}{1-15}{}
\tauAr{1-5}{1-3}{}
\tauAr{1-7}{1-5}{}
\tauAr{1-9}{1-7}{}
\tauAr{1-11}{1-9}{}
\tauAr{1-13}{1-11}{}
\tauAr{1-15}{1-13}{}
% \tauAr{2-4}{2-2}{}
% \tauAr{2-6}{2-4}{}
\tauAr{2-8}{2-6}{}
\tauAr{2-10}{2-8}{}
\tauAr{2-12}{2-10}{}
% \tauAr{2-14}{2-12}{}
% \tauAr{1-3}{1-1}{}
% \tauAr{1-5}{1-3}{}
\tauAr{3-7}{3-5}{}
% \tauAr{3-9}{3-7}{}
% \tauAr{3-11}{3-9}{}
% \tauAr{3-13}{3-11}{}
\end{tikzcd}
$$

This gives an example where the corresponding DG-algebra is not $2$-self-injective. 
\end{example}

% \begin{lemma}
% \label{chara:d_mono_chara_three}
% Consider the following bicartesian diagram in $\scrT$:
% $$\begin{tikzcd}
% X \ar[r,"\iota"] \ar[d,"f"'] & C \ar[d,"\pi"] \\
% Y \ar[r,"g"'] & Z
% \Ar{1-1}{2-2}{"h", red}
% \end{tikzcd}$$
% where $C\in\scrC$. Then the following conditions are equivalent:
% \begin{enumerate}
% 	\item $Qf$ is a $d$-monomorphism in $\scrT\lquot\scrC$.
% 	\item The induced morphism 
% 	$$(g\circ-, \pi\circ-)\colon \calT(*,X)|_{\calC}\oplus \calT(*,C)|_{\calC} \to \calT(*,Y)|_{\calC}$$
% 	is an epimorphism in $\Mod \calC$. 
% \end{enumerate}
% \end{lemma}
% \begin{proof}
% (1) $\Rightarrow$ (2): 
% \end{proof}

\appendix

\section{\texorpdfstring{Morita theory for abelian $d$-truncated DG-categories}{Morita theory for abelian d-truncated DG-categories}}

In this appendix, we give a Morita-theoretic characterization of $d$-extended module categories in abelian $d$-truncated DG-categories. The main result is Theorem~\ref{thm:morita}. We fix a commutative ring $k$ with $\bbU$-small underlying set. We also fix a $\bbV$-small connective DG-category $\scrA$ with $\bbU$-small and $K$-projective hom-complexes. Let $d$ be a positive integer. 

We note that the derived DG-category $\scrD_{\bbV}(\scrA)$, which we simply denote by $\scrD(\scrA)$, of $\scrA$ is $\bbW$-small and locally $\bbV$-small. Furthermore, $\per\dg\scrA$ can be constructed as the DG-category of one-sided twisted complexes on a DG-category that formally adds finite direct sums and direct summands of $\scrA$, so $\per\dg\scrA$ is a $\bbV$-small and locally $\bbU$-small DG-category. These DG-categories are locally $K$-projective by construction, since $\scrA$ is locally $K$-projective.

Hence, we often use representable right DG-modules $\scrA(-,A)$ for $A\in\scrA$, and we denote $\scrA(-,A)$ by $A$ for simplicity. In this identification, we regard $\scrA$ as a full DG-subcategory of $\per\dg\scrA$ via the Yoneda embedding. To avoid confusion, we denote the image of $A\in\scrA\op$ under the Yoneda embedding $\scrA\op\to \per\dg\scrA\op$ by $A\hut$. 

\subsection{\texorpdfstring{$(d+1)$}{(d+1)}-term complexes and their \texorpdfstring{$d$}{d}-iterated kernels}

In this subsection, we introduce the notion of $d$-iterated kernels and $d$-iterated cokernels for $(d+1)$-term complexes in $\scrA$. This notion recovers the usual kernel and cokernel in a connective DG-category (see Definition~\ref{def:exactness} and \cite[Definition~3.49]{che2023}) when $d=1$. 

\begin{definition}
\label{def:complex}
Define the full sub DG-category $\scrK^{[0,d]}(\scrA)\subset\tau^{\leq 0}\bigl(\per\dg\scrA\bigr)$ as follows:
$$\scrK^{[0,d]}(\scrA):=\scrA[-d]*\scrA[-d+1]*\cdots*\scrA[-1]*\scrA.$$
Similarly, we define $\scrK^{[-d,0]}(\scrA)$ by
$$\scrK^{[-d,0]}(\scrA):=\scrA*\scrA[1]*\cdots*\scrA[d-1]*\scrA[d].$$ 
where $\scrA[i]$ is the full DG-subcategory of $\per\dg\scrA$ consisting of objects $A[i]$ for $A\in\scrA$. The notation $*$ is defined in Definition~\ref{def:ast_sub}.
\end{definition}

% \begin{remark}
% Since we can take a DG-enhancement of $\per\scrA$ as one-sided 
% \end{remark}

\begin{lemma}
\label{lem:dual_of_cpx}
The equivalence functor $(-)\hut:=\rhom{\scrA}{-,\scrA}:(\per\dg\scrA)\op\to\per\dg{\scrA\op}$ induces an equivalence 
$$\bigl(\calK^{[0,d]}(\scrA)\bigr)\op\simeq\calK^{[-d,0]}(\scrA\op).$$
In particular, we also have a quasi-equivalence $\scrK^{[0,d]}(\scrA)\simeq\scrK^{[-d,0]}(\scrA\op)$.
\end{lemma}
\begin{proof}
It immediately follows from the definition of $\scrK^{[0,d]}(\scrA)$ and $\scrK^{[-d,0]}(\scrA\op)$ by noting that $\rhom{\scrA}{\scrA[i],\scrA}\cong\rhom{\scrA}{\scrA,\scrA}[-i]\cong \scrA[-i]$ for any $i\in\bbZ$.
\end{proof}

\begin{remark}
Since we denote the image of $A\in\scrA\op$ under the Yoneda embedding $\scrA\op\to \per\dg\scrA\op$ by $A\hut$, the above notation $(-)\hut$ is compatible with our notation for representable left DG-modules. 
\end{remark}

\begin{definition}
\label{def:iterated_kernel}
Let $C^\bullet\in \scrK^{[0,d]}(\scrA)$. The \emph{$d$-iterated kernel} of $C^\bullet$ is defined as the object $\Kerd C^\bullet\in\scrA$ such that there exists an isomorphism
$$\scrA(-,\Kerd C^\bullet)\cong \tau^{\leq 0} C^\bullet$$
in $\der{\scrA}$. By definition, $d$-iterated kernels are unique up to isomorphisms in $\scrA$. 

If for any $C^\bullet\in \scrK^{[0,d]}(\scrA)$, the $d$-iterated kernel $\Kerd C^\bullet$ exists, we say that \emph{$\scrA$ has $d$-iterated kernels}.

Dually, for $C^\bullet\in \scrK^{[-d,0]}(\scrA)$, the \emph{$d$-iterated cokernel} of $C^\bullet$ is defined as the object $\Cokd C^\bullet\in\scrA$ such that there exists an isomorphism
$$\scrA(\Cokd C^\bullet,-)\cong \tau^{\leq 0} (C^\bullet\hut)$$
in $\der{\scrA\op}$. By definition, $d$-iterated cokernels are unique up to isomorphisms in $\scrA$.

If for any $C^\bullet\in \scrK^{[-d,0]}(\scrA)$, the $d$-iterated cokernel $\Cokd C^\bullet$ exists, we say that \emph{$\scrA$ has $d$-iterated cokernels}.
\end{definition}

\begin{remark}
\label{rem:dual_ker_coker}
By Lemma \ref{lem:dual_of_cpx}, we can check that $d$-iterated cokernels in $\scrA$ are precisely the $d$-iterated kernels in $\scrA\op$. 
\end{remark}

\begin{remark}
\label{rem:ker_mor}
Assume that $\scrA$ has $d$-iterated kernels. Then the bimodule 
$$\tau^{\leq 0}(-)\in \der{(\scrK^{[0,d]}(\scrA))\op\lox \scrA}$$ 
is right quasi-representable by definition. Hence, by Theorem \ref{thm:internal_hom} (\cite[Theorem~6.1]{toe2007}), this bimodule induces the following morphism in $\hqe_\bbV$:
$$\Kerd\colon\scrK^{[0,d]}(\scrA)\to \scrA.$$
Dually, if $\scrA$ has $d$-iterated cokernels, we also have a morphism $\Cokd\colon\scrK^{[-d,0]}(\scrA)\to \scrA$ in $\hqe_\bbV$.
\end{remark}

To describe $d$-iterated kernels more explicitly, we need the following lemma and its dual version.

\begin{lemma}
\label{lem:step1}
Let $\calT$ be a triangulated category and $(\calT^{\leq 0},\calT^{\geq 0})$ be a $t$-structure on $\calT$. Consider the following diagram in $\calT$:
$$
\begin{tikzcd}
\Cocone f' & M & \tau^{\leq 0} N & \Cocone f' [1]\\
\Cocone f & M & N & \Cocone f [1] 
\Ar{1-1}{1-2}{}
\Ar{1-2}{1-3}{"{f'}"}
\Ar{1-3}{1-4}{}
\Ar{2-1}{2-2}{}
\Ar{2-2}{2-3}{"{f}"}
\Ar{2-3}{2-4}{}
\Ar{1-1}{2-1}{}
\Ar{1-2}{2-2}{equal}
\Ar{1-3}{2-3}{}
\Ar{1-4}{2-4}{}
\end{tikzcd}
$$
where the rows are distinguished triangles and $M\in \calT^{\leq 0}$. Then the natural morphism 
$$\tau^{\leq 0} \Cocone f'\to \Cocone f' \to \Cocone f$$ 
induces an isomorphism $\tau^{\leq 0} \Cocone f'\cong \tau^{\leq 0} \Cocone f$.
\end{lemma}
\begin{proof}
It is easy to show by using octahedral axiom.
\end{proof}

The following lemma is a dual version of Lemma \ref{lem:step1}.

\begin{lemma}
\label{lem:step1op}
Let $\calT$ be a triangulated category and $(\calT^{\leq 0},\calT^{\geq 0})$ be a $t$-structure on $\calT$. Consider the following diagram in $\calT$:
$$
\begin{tikzcd}
\Cone f[-1] & M & N & \Cone f\\
\Cone f'[-1] & \tau^{> -d}M &  N & \Cone f'
\Ar{1-1}{1-2}{}
\Ar{1-2}{1-3}{"{f}"}
\Ar{1-3}{1-4}{}
\Ar{2-1}{2-2}{}
\Ar{2-2}{2-3}{"{f'}"}
\Ar{2-3}{2-4}{}
\Ar{1-1}{2-1}{}
\Ar{1-2}{2-2}{}
\Ar{1-3}{2-3}{equal}
\Ar{1-4}{2-4}{}
\end{tikzcd}
$$
where the rows are distinguished triangles and $N\in \calT^{>-d}$. Then the natural morphism 
$$\Cone f\to \Cone f' \to \tau^{> -d} \Cone f'$$
induces an isomorphism $\tau^{> -d} \Cone f\cong \tau^{> -d} \Cone f'$.
\end{lemma}

The following proposition gives a more explicit description of $d$-iterated kernels.

\begin{proposition}
\label{prop:iteraded_ker}
Assume that $\scrA$ has kernels. Consider the following triangles in $\per\scrA$:
$$M^{[d-i,d]}\xrightarrow{f^{d-i}} M^{d-i}\xrightarrow{g^{d-(i-1)}} M^{[d-(i-1),d]}\dashrightarrow$$
for $1\leq i\leq d$. Denote $M^{[d,d]} = M^{d}$. Assume that $M^{d-i}\in\scrA$ for any $0\leq i\leq d$.

Then, we can construct the following commutative diagram in $\per\scrA$ for each $0\leq i\leq d$:
$$
\begin{tikzcd}
& M^{[d-i,d]} & \\
M^{d-(i+1)} && M^{d-i}\\
& K^{d-i} & 
\Ar{2-1}{1-2}{"{g^{d-i}}"}
\Ar{1-2}{2-3}{"{f^{d-i}}"}
\Ar{2-1}{3-2}{"{\varphi^{d-i}}"'}
\Ar{3-2}{2-3}{"{\psi^{d-i}}"'}
\Ar{3-2}{1-2}{"{\theta^{d-i}}"'}
\end{tikzcd}
$$
% $$
% \begin{tikzcd}
% M^{[d-i,d]} & \\
% & M^{d-i}\\
% K^{d-i} &
% \Ar{1-1}{2-2}{"{f^{d-i}}"}
% \Ar{3-1}{2-2}{"{\varphi^{d-i}}"'}
% \Ar{3-1}{1-1}{"{\theta^{d-i}}"'}
% \end{tikzcd}
% $$
which satisfies the following conditions:
\begin{enumerate}
	\item In the case $i=0$, $K^{d}:=M^{d}$, $\theta^{d}=\id_{M^d}$ and $\varphi^{d}:=g^d$. 
	\item In the case $i=d$, $M^{-1}:=0$, $g^{-1}:=0$ and $\varphi^{-1}:=0$.
	\item $K^{d-i}\in\scrA$ and it is isomorphic to $\tau^{\leq 0}\Cocone (\varphi^{d-(i-1)})$ for $1\leq i\leq d$.
	\item $\theta^{d-i}$ induces an isomorphism $K^{d-i}\cong \tau^{\leq 0} M^{[d-i,d]}$ for any $1\leq i\leq d$.
\end{enumerate}
% where $K^{d-i}$ is defined to be the $\tau^{\leq 0}\Cocone \varphi^{d-(i-1)}$ and $\theta^{d-i}$ induces an isomorphism $\tau^{\leq 0}M^{[d-i,d]}\cong K^{d-i}$. For $i=d$, we have the right side of the above diagram. In particular, we have $\theta^{0}\colon M^{[0,d]}\cong K^0$ and it induces an isomorphism $\tau^{\leq 0}M^{[0,d]}\cong K^0$.
% $$
% \begin{tikzcd}[column sep=-2mm]
% M^{[0,d]} && M^{[1,d]} && \cdots && \cdots && M^{[d-i,d]} && \cdots && M^{[d-1,d]} && M^d & \\
% & M^{0} && M^{1} && \cdots && M^{d-(i+1)} && M^{d-i} && \cdots && M^{d-1} && M^d\\
% K^d && K^{d-1} && \cdots && \cdots && K^i && \cdots && K^1 && K^0
% \Ar{1-1}{2-2}{"{f^0}"}
% \Ar{2-2}{1-3}{"{g^0}"}
% \Ar{1-3}{2-4}{"{f^1}"}
% \Ar{2-4}{1-5}{}
% % \Ar{1-5}{2-6}{}
% % \Ar{2-6}{1-7}{}
% \Ar{1-7}{2-8}{}
% \Ar{2-8}{1-9}{"{g^{i-1}}"}
% \Ar{1-9}{2-10}{"{f^i}"}
% \Ar{2-10}{1-11}{}
% % \Ar{1-11}{2-12}{}
% % \Ar{2-12}{1-13}{}
% \Ar{1-13}{2-14}{}
% \Ar{2-14}{1-15}{"{g^{d-1}}"}
% \Ar{1-15}{2-16}{"{f^d}"}
% \Ar{3-1}{2-2}{"{\varphi^d}"'}
% \Ar{2-2}{3-3}{"{\psi^{d-1}}"'}
% \Ar{3-3}{2-4}{"{\varphi^{d-1}}"'}
% \end{tikzcd}
% $$
\end{proposition}
\begin{proof}
We construct the above diagram by induction on $i$. The case of $i=0$ is trivial. Assume that we have constructed the above diagram for $i-1$. Then we have the following commutative diagram in $\der{\scrA}$:
$$
\begin{tikzcd}
\Cocone \varphi^{d-(i-1)} & M^{d-i} & K^{d-(i-1)} & \Cocone \varphi^{d-(i-1)} [1]\\
M^{[d-i,d]} & M^{d-i} & M^{[d-(i-1),d]} & M^{[d-i,d]} [1] 
\Ar{1-1}{1-2}{}
\Ar{1-2}{1-3}{"{\varphi^{d-(i-1)}}"}
\Ar{1-3}{1-4}{}
\Ar{2-1}{2-2}{"{f^{d-i}}"'}
\Ar{2-2}{2-3}{"{g^{d-(i-1)}}"'}
\Ar{2-3}{2-4}{}
\Ar{1-1}{2-1}{}
\Ar{1-2}{2-2}{equal}
\Ar{1-3}{2-3}{"{\theta^{d-(i-1)}}"'}
\Ar{1-4}{2-4}{}
\end{tikzcd}
$$

Since $\theta^{d-(i-1)}$ induces an isomorphism $K^{d-(i-1)}\cong \tau^{\leq 0} M^{[d-(i-1),d]}$ and $M^{d-i}\in \scrA$, we can apply Lemma \ref{lem:step1} to the above diagram and obtain an isomorphism $\tau^{\leq 0} \Cocone \varphi^{d-(i-1)}\cong \tau^{\leq 0} M^{[d-i,d]}$. Define $K^{d-i}:=\tau^{\leq 0} \Cocone \varphi^{d-(i-1)}$ and $\theta^{d-i}$ to be the composition of 
$$K^{d-i}\cong \tau^{\leq 0} \Cocone \varphi^{d-(i-1)}\cong \tau^{\leq 0} M^{[d-i,d]}\to M^{[d-i,d]}.$$
Since $M^{d-i}$ and $K^{d-(i-1)}$ are in $\scrA$ and $\scrA$ has kernels, we have $K^{d-i}\in\scrA$.  
We also define $\psi^{d-i}:=f^{d-i}\circ\theta^{d-i}$. Since $\theta^{d-(i-1)}$ induces an isomorphism $K^{d-(i-1)}\cong \tau^{\leq 0} M^{[d-(i-1),d]}$ and $M^{d-(i+1)}\in \scrA$, we have a unique morphism $\varphi^{d-i}\colon M^{d-(i+1)}\to K^{d-i}$ such that $\theta^{d-i}\circ \varphi^{d-i}= g^{d-i}$. Thus, we have the following commutative diagram in $\der{\scrA}$: 
$$
\begin{tikzcd}
& M^{[d-i,d]} & \\
M^{d-(i+1)} && M^{d-i}\\
& K^{d-i} &
\Ar{2-1}{1-2}{"{g^{d-i}}"}
\Ar{1-2}{2-3}{"{f^{d-i}}"}
\Ar{2-1}{3-2}{"{\varphi^{d-i}}"'}
\Ar{3-2}{2-3}{"{\psi^{d-i}}"'}
\Ar{3-2}{1-2}{"{\theta^{d-i}}"'}
\end{tikzcd}
$$
This completes the induction step and we have constructed the desired diagram for any $0\leq i\leq d$.
\end{proof}

The following corollary means that $d$-iterated kernels are obtained by iterating kernels $d$ times.

\begin{cor}
\label{cor:iterated_ker_exists}
If $\scrA$ has kernels, then $\scrA$ has $d$-iterated kernels.
\end{cor}
\begin{proof}
Let $C^\bullet\in \scrK^{[0,d]}(\scrA)$. By definition of $\scrK^{[0,d]}(\scrA)$, we have the following triangles in $\per\scrA$:
$$C^{[d-i,d]}\to C^{d-i}\to C^{[d-(i-1),d]}\dashrightarrow$$
for $1\leq i\leq d$ where $C^{[d,d]}:=C^d$, $C^{d-i}\in\scrA$ for any $0\leq i\leq d$ and $C^\bullet=C^{[0,d]}$. By Proposition \ref{prop:iteraded_ker}, we can construct the $\theta^{0}\colon K^0\to C^{[0,d]}$ which induces an isomorphism $\tau^{\leq 0}C^{[0,d]}\cong K^0$. Thus, $K^0\in\scrA$ is the $d$-iterated kernel of $C^\bullet$.
\end{proof}

\subsection{Morita theorem}

In this subsection, we fix a $\bbV$-small abelian $d$-truncated DG-category $\scrA$ with $\bbU$-small and $K$-projective hom-complexes. We give a Morita-theoretic characterization of $d$-extended module categories. 

First, we recall the projective objects in $\scrA$. Projective objects are defined as projective objects in the natural exact structure on $\scrA$ as follows:
\begin{proposition}[{\cite[Proposition~3.14]{moc2025}}]
\label{prop:exact_dg_abel}
By taking all short exact sequences as conflations, $\scrA$ naturally becomes an exact DG-category. In particular, $H^0(\scrA)$ naturally becomes an extriangulated category. Here, deflation in $\scrA$ is $d$-epimorphism and inflation is $d$-monomorphism.
\end{proposition}

\begin{definition}
% [{\cite[]{che2024}}]
\label{def:proj}
An object $P\in\scrA$ is called \emph{projective} if it is a projective object in the extriangulated category $H^0(\scrA)$. We denote by $\scrP$ the full sub DG-category of $\scrA$ consisting of projective objects.
\end{definition}

The following lemma means that the DG-functor $\scrA(P,-)\colon \scrA\to \der{k}$ ``preserves'' $d$-epimorphisms.

\begin{lemma}
\label{lem:ex_proj}
Let $P\in \scrA$ be a projective object. Consider the following exact sequence in $\scrA$:
$$X\to Y\to Z\quad (h\colon X\to Z)$$
Then we have the following triangle in $\der{k}$:
$$\scrA(P,X)\to \scrA(P,Y)\to \scrA(P,Z)\to \scrA(P,X)[1].$$
\end{lemma}
\begin{proof}
Since $P$ is projective, we have a surjective morphism $\calA(P,Y)\to \calA(P,Z)$ in $\Mod k$. Hence, we can see $H^1\bigl(\Cocone(\scrA(P,X)\to \scrA(P,Y))\bigr)=0$. This implies that the above sequence is a triangle in $\der{k}$.
\end{proof}

\begin{lemma}
\label{lem:chara_1_epi}
Consider the following exact sequence in $\scrA$:
$$K\to X\xrightarrow{f} Y \quad (h\colon K\to Y)$$
Then $f$ is a $1$-epimorphism if and only if $\Sigma K=0$ in $\calA$.
\end{lemma}
\begin{proof}
By \cite[Proposition~2.18]{moc2025}, we have the right triangulated structure $(\Sigma, \nabla)$ on $\calA$. Consider the following right triangle in $\calA$:
$$K\to X\to Y\to \Sigma K.$$
By rotating the above right triangle, we also have the following right triangle in $\calA$:
$$X\to Y\to \Sigma K\to \Sigma X.$$
Thus, we have the following long exact sequence:
$$\calA(\Sigma Y,-)\to \calA(\Sigma X,-)\to \calA(\Sigma K, -)\to \calA(Y, -)\to \calA(X, -)$$
Thus, if $f$ is a $1$-epimorphism, we have $\Sigma K=0$ in $\calA$. Conversely, if $\Sigma K=0$ in $\calA$, we have $\Sigma^i K=0$ for any $1\leq i\leq d-1$. Hence, we have $\calA(\Sigma^i K,-)=0$ for any $1\leq i\leq d-1$. This implies that the morphism $\calA(Y,-)\to \calA(X,-)$ is injective and $\calA(\Sigma^i Y,-)\to \calA(\Sigma^i X,-)$ is an isomorphism for $1\leq i\leq d-1$. Thus, $f$ is a $1$-epimorphism.
\end{proof}

In Lemma~\ref{lem:ex_proj}, we have seen that the DG-functor $\scrA(P,-)\colon \scrA\to \der{k}$ preserves $d$-epimorphisms. The following proposition means that $\scrA(P,-)$ also preserves $1$-epimorphisms. 

\begin{proposition}
\label{lem:pres_1_epi}
Let $P$ be a projective object in $\scrA$. If $f\colon X\to Y$ is a $1$-epimorphism in $\scrA$, then the morphism $\scrA(P,X)\to \scrA(P,Y)$ induces a surjective morphism $H^{-d+1}(\scrA(P,X))\to H^{-d+1}(\scrA(P,Y))$ and an isomorphism $H^i(\scrA(P,X))\cong H^i(\scrA(P,Y))$ for any $i>-d+1$.
\end{proposition}

To prove this, we show some technical lemmata.

\begin{lemma}
\label{lem:truncation}
For any $K\in\scrA$, we have the following conflation:
$$\Sigma^{d-1}\Omega^{d-1}K\to K \to \Omega \Sigma K$$ 
\end{lemma}
\begin{proof}
By using \cite[Theorem~5.4.]{plogmann2026}, we have a quasi-fully faithful morphism $\scrA\to \scrT$ in $\hqe$ where $\scrT$ is a stable DG-category with $t$-structure $(\scrT^{\leq 0}, \scrT^{\geq 0})$, and the essential image of $\scrA$ in $\scrT$ corresponds to the $d$-extended heart $\scrT^{[-d+1,0]}$ (see \cite[Definition~3.40]{moc2025}). In this situation, there exists the following conflation in $\scrA$:
$$\tau^{\leq-d}K\to K \to \tau^{> -d} K$$
where $\tau^{\leq -d}$ and $\tau^{> -d}$ are the truncation functors associated with the $t$-structure on $\scrT$. By \cite[Theorem~3.50]{moc2025}, we have $\tau^{\leq -d}K\cong \Sigma^{d-1}\Omega^{d-1}K$ and $\tau^{> -d}K\cong \Omega \Sigma K$ in $\calA$. Thus, we have the desired conflation.
\end{proof}

\begin{lemma}
\label{lem:chara_sig_zero}
For any $K\in\scrA$, we have $\Sigma K=0$ in $\calA$ if and only if the morphism $\Sigma^{d-1}\Omega^{d-1}K\to K$ in Lemma \ref{lem:truncation} is an isomorphism.
\end{lemma}
\begin{proof}
It immediately follows from the embedding to the stable DG-category $\scrT$ in the proof of Lemma \ref{lem:truncation}.
\end{proof}

\begin{lemma}
Let $P$ be a projective object in $\scrA$. If $K$ satisfies $\Sigma K=0$ in $\calA$, then we have $H^{i}(\scrA(P,K))=0$ for any $i>-d+1$.
\end{lemma}
\begin{proof}
Since $\Sigma K=0$ in $\calA$, we have $\Sigma^{d-1}\Omega^{d-1}K\cong K$ by Lemma \ref{lem:chara_sig_zero}. Hence, we have 
$$H^i(\scrA(P,K))\cong H^i(\scrA(P,\Sigma^{d-1}\Omega^{d-1}K))\cong H^0(\scrA(P,\Omega^{-i}\Sigma^{d-1}\Omega^{d-1}K))$$
for any $-d+1< i \leq 0$. Since $d-1+i>0$, we have $\Omega^{-i}\Sigma^{d-1}\Omega^{d-1}K\cong \Sigma^{d-1+i}\Omega^{d-1}K$ and thus, 
the morphism $0\to \Sigma^{d-1-i}\Omega^{d-1}K$ is an $d$-epimorphism. Thus, the morphism 
$$H^0(\scrA(P,0))\to H^0(\scrA(P,\Sigma^{d-1-i}\Omega^{d-1}K))$$
is surjective and we have $H^0(\scrA(P,\Sigma^{d-1-i}\Omega^{d-1}K))=0$. This completes the proof.
\end{proof}

\begin{proof}[Proof of Proposition~\ref{lem:pres_1_epi}]
Let $f$ be a $1$-epimorphism and $K$ be a kernel of $f$. Then we have $\Sigma K=0$ in $\calA$ by Lemma \ref{lem:chara_1_epi}. Hence, we have $H^i(\scrA(P,K))=0$ for any $-d+1< i$ by the above lemma. Thus, the morphism $\scrA(P,X)\to \scrA(P,Y)$ induces a surjective morphism $H^{-d+1}(\scrA(P,X))\to H^{-d+1}(\scrA(P,Y))$ and an isomorphism $H^i(\scrA(P,X))\cong H^i(\scrA(P,Y))$ for any $i>-d+1$.
\end{proof}

\begin{lemma}
\label{lem:right_ex_proj}
Let $P$ be a projective object in $\scrA$. Consider the following right exact sequence in $\scrA$:
$$X\xrightarrow{f} Y\xrightarrow{g} Z\quad (h\colon X\to Z)$$
Then we have the following isomorphism in $\der{k}$:
$$\tau^{>-d}\Cone\bigl(\scrA(P,X)\to \scrA(P,Y)\bigr)\cong \scrA(P,Z).$$
\end{lemma}
\begin{proof}
By \cite[Theorem~3.21.]{moc2025}, we can take a factorization of $f\colon X\to Y$ as $X\xrightarrow{e_1} \Im f\xrightarrow{m_d} Y$ in $H^0(\scrA)$ where $e_1$ is a $1$-epimorphism and $m_d$ is a $d$-monomorphism. In this case, we have the following exact sequence in $\scrA$ (see \cite[Proposition~3.22.]{moc2025}):
$$\Im f\xrightarrow{m_d} Y\xrightarrow{g} Z\quad (h'\colon \Im f\to Z).$$
By Lemma \ref{lem:ex_proj}, we have the following triangle in $\der{k}$:
$$\scrA(P,\Im f)\to \scrA(P,Y)\to \scrA(P,Z)\to \scrA(P,\Im f)[1].$$
On the other hand, the morphism $\scrA(P,X)\to \scrA(P,\Im f)$ induces an isomorphism $H^i(\scrA(P,X))\cong H^i(\scrA(P,\Im f))$ for any $i>-d+1$ and a surjective morphism $H^{-d+1}(\scrA(P,X))\to H^{-d+1}(\scrA(P,\Im f))$ by Lemma \ref{lem:pres_1_epi}. Hence, the cone $\Cone(\scrA(P,X)\to \scrA(P,\Im f))$ is in $\calD^{\leq -d+1}(\scrA)$. By the octahedral axiom, we have the following triangle in $\der{k}$:
$$\Cone\bigl(\scrA(P,X)\to \scrA(P,\Im f)\bigr)\to \Cone\bigl(\scrA(P,X)\to \scrA(P,Y)\bigr)\to \scrA(P,Z)\rightdasharrow.$$
Thus, we have $\tau^{>-d}\Cone\bigl(\scrA(P,X)\to \scrA(P,Y)\bigr)\cong \scrA(P,Z)$.
\end{proof}

\begin{proposition}
\label{prop:proj_d_cok}
Let $M^\bullet$ be an object in $\scrK^{[-d,0]}(\scrA)$. Then we have the following isomorphism in $\der{\scrP}$:
$$\scrA(-,\Cokd M^\bullet)|_{\scrP}\cong \tau^{>-d}M^\bullet|_{\scrP}.$$
\end{proposition}
\begin{proof}
Take the following diagram in $\per\scrA$:
$$
\begin{tikzcd}[column sep={12mm,between origins},row sep={13mm}]
M^{[-d,-d]} && M^{[-d,-d+1]} && \cdots &&  \cdots && M^{[-d,-1]} && M^{[-d,0]}\\
& M^{-d+1} && M^{-d+2} && \cdots && \cdots && M^{0} &
\Ar{1-1}{2-2}{"f^{-d}"'}
\Ar{2-2}{1-3}{"g^{-d+1}"}
\Ar{1-3}{2-4}{"f^{-d+1}"'}
\Ar{2-4}{1-5}{"g^{-d+2}"}
\Ar{2-8}{1-9}{"g^{-1}"}
\Ar{1-9}{2-10}{"f^{-1}"'}
\Ar{2-10}{1-11}{"g^0"}
\end{tikzcd}
$$
where $M^{[-d,-d+(i-1)]}\xrightarrow{f^{-d+(i-1)}} M^{-d+i}\xrightarrow{g^{-d+i}} M^{[-d,-d+i]}\dashrightarrow$ is a triangle in $\per\scrA$ for any $1\leq i\leq d$, $M^\bullet=M^{[-d,0]}$, and $M^{-d}:=M^{[-d,-d]} ,M^{-d+i}\in\scrA$ for any $1\leq i\leq d$. Then, by Proposition~\ref{prop:iteraded_ker}, we can construct the following commutative diagram in $\der{\scrA\op}$ for each $0\leq i\leq d$:

$$
\begin{tikzcd}
& M^{[-d,-d+i]}\hut & \\
M^{-d+i}\hut && M^{-d+i+1}\hut\\
& K^{-d+i}&
\Ar{1-2}{2-1}{"{g^{-d+i}\hut}"'}
\Ar{2-3}{1-2}{"{f^{-d+i}\hut}"'}
\Ar{3-2}{2-1}{"{\varphi^{-d+i}}"}
\Ar{2-3}{3-2}{"{\psi^{-d+i}}"}
\Ar{3-2}{1-2}{"{\theta^{-d+i}}"}
\end{tikzcd}
$$
which satisfies the following conditions:
\begin{enumerate}
	\item In the case $i=0$, $K^{-d}:=M^{-d}\hut$,  $\theta^{-d}=\id_{K^{-d}}$ and $\varphi^{-d}:=f^{-d}\hut$.
	\item In the case $i=d$, $K^{1}:=0$, $g^{1}:=0$ and $f^0:=0$.
	\item $K^{-d+i}$ is isomorphic to $\tau^{\leq 0}\Cocone (\varphi^{-d+(i-1)})$ for $1\leq i\leq d$.
	\item $\theta^{-d+i}$ induces an isomorphism $K^{-d+i}\cong \tau^{>-d} M^{[-d,-d+i]}\hut$ for any $1\leq i\leq d$.
\end{enumerate}
Since $\scrA$ has cokernels, we have $K^{-d+i}\in\scrA\op$ for any $1\leq i\leq d$. In particular, we have $K^0\cong (\Cokd M^\bullet)\hut$ by definition of $d$-iterated cokernel. Take $C^{-d+i}\in\scrA$ such that $C^{-d+i}\cong K^{-d+i}\hut$ for $1\leq i\leq d$. 

In this situation, we have the following right exact sequence in $\scrA$ by the condition (3) above:
$$C^{-d+(i-1)}\xrightarrow{p^{-d+(i-1)}} M^{-d+i}\xrightarrow{q^{-d+i}} C^{-d+i}\quad (h\colon C^{-d+(i-1)}\to C^{-d+i}).$$
for any $1\leq i\leq d$. By Lemma~\ref{lem:right_ex_proj}, we have an isomorphism 
$$\tau^{>-d}\Cone\bigl(C^{-d+(i-1)}|_{\scrP}\to M^{-d+i}|_{\scrP}\bigr)\cong C^{-d+i}|_{\scrP}$$
for any $1\leq i\leq d$. Then, by using Lemma~\ref{lem:step1op} inductively, we can construct the following commutative diagram in $\der{\scrP}$ for each $0\leq i\leq d$:

$$
\begin{tikzcd}
& M^{[-d,-d+i]}|_{\scrP} & \\
M^{-d+i}|_{\scrP} && M^{-d+i+1}|_{\scrP}\\
& C^{-d+i}|_{\scrP}&
\Ar{2-1}{1-2}{"{g^{-d+i}|_{\scrP}}"}
\Ar{1-2}{2-3}{"{f^{-d+i}|_{\scrP}}"}
\Ar{2-1}{3-2}{"{q^{-d+i}|_{\scrP}}"'}
\Ar{3-2}{2-3}{"{p^{-d+i}|_{\scrP}}"'}
\Ar{1-2}{3-2}{"{\zeta^{-d+i}}"'}
\end{tikzcd}
$$
which satisfies the following conditions:
\begin{enumerate}
	\item In the case $i=0$, $C^{-d} :=K^{-d}\hut$, $\zeta^{-d}=\id_{C^{-d}}$ and $q^{-d}:=g^{-d}$.
	\item In the case $i=d$, $C^{1}:=0$, $g^{1}:=0$ and $f^0:=0$.
	\item $C^{-d+i}|_{\scrP}$ is isomorphic to $\tau^{>-d}\Cone (p^{-d+(i-1)}|_{\scrP})$ for $1\leq i\leq d$.
	\item $\zeta^{-d+i}$ induces an isomorphism $C^{-d+i}|_{\scrP}\cong \tau^{>-d} M^{[-d,-d+i]}|_{\scrP}$ for any $1\leq i\leq d$.
\end{enumerate}
% This construction is obtained by applying Lemma~\ref{lem:step1op} inductively by same way as in the proof of Proposition~\ref{prop:iteraded_ker}. 
In particular, we have $C^0|_{\scrP}\cong \tau^{>-d} M^{[-d,0]}|_{\scrP}=\tau^{>-d} M^\bullet|_{\scrP}$. On the other hand, we also have $C^0\cong \Cokd M^\bullet$. Thus, we have $\scrA(-,\Cokd M^\bullet)|_{\scrP}\cong \tau^{>-d} M^\bullet|_{\scrP}$.
% By applying Lemma~\ref{lem:right_ex_proj} inductively, the object $C^{-d+i}|_{\scrP}\cong \tau^{>-d} \bigl(\Cone (\psi^{-d+(i-1)})\bigr)|_{\scrP}$ is isomorphic to representables for any $1\leq i\leq d$. 
%  $g^{-d}=\id_{C^{-d}}$ and $\varphi^{-d}=\id_{C^{-d}}$ for $i=0$ and $C^{-1}:=0$, $g^{-1}:=0$ and $\phi^{-1}:=0$ for $i=d$. 
% $$C^{[-d,-d+i]}\to C^{-d+i}\to C^{[-d,-d+(i-1)]}\dashrightarrow$$
% for $1\leq i\leq d$ where $C^{[-d,-d]}:=C^{-d}$, $C^{-d+i}\in\scrA$ for any $0\leq i\leq d$ and $C^\bullet=C^{[-d,0]}$. Then we can construct the following commutative diagram in $\der{\scrP}$ for each $0\leq i\leq d$:
% $$
% \begin{tikzcd}
% & M^{[d-i,d]} & \\
% M^{d-(i+1)} && M^{d-i}\\
% & K^{d-i} &
% \Ar{2-1}{1-2}{"{g^{d-i}}"}
% \Ar{1-2}{2-3}{"{f^{d-i}}"}
% \Ar{2-1}{3-2}{"{\varphi^{d-i}}"'}
% \Ar{3-2}{2-3}{"{\psi^{d-i}}"'}
% \Ar{3-2}{1-2}{"{\theta^{d-i}}"'}
% \end{tikzcd}
% $$
% \textcolor{blue}{need to add.}
\end{proof}

\begin{lemma}
\label{lem:claim1}
Let $P^\bullet, Q^\bullet$ be objects in $\scrK^{[-d,0]}(\scrP)\subset\scrK^{[-d,0]}(\scrA)$. Then we have the following isomorphism in $\der{k}$:
$$\rhom{\scrA}{-[d],\tau^{\leq-d}Q^\bullet}|_{\scrP}\lox_{\scrP}\rhom{\scrA}{P^\bullet,-[d]}|_{\scrP}\xrightarrow{m}\rhom{\scrA}{P^\bullet, \tau^{\leq-d}Q^\bullet}$$
where $m$ is the morphism induced by the composition of morphisms in $\der{\scrA}$.
\end{lemma}
\begin{proof}
Since $P^\bullet\in \scrK^{[-d,0]}(\scrP)$, it is enough to check that the morphism $m$ induces an isomorphism in the case $P^\bullet=P[i]$ with $P\in\scrP$ and $0\leq i\leq d$. In this case, the morphism $m$ is given by the following composition of morphisms in $\der{\scrA}$:
\begin{align*}
\rhom{\scrA}{-[d],\tau^{\leq -d}Q^\bullet}|_{\scrP}&\lox_{\scrP}\rhom{\scrA}{P[i],-[d]}|_{\scrP}\\
&\cong \rhom{\scrA}{-[d],\tau^{\leq-d}Q^\bullet}|_{\scrP}\lox_{\scrP}\rhom{\scrP}{P[i],-[d]}\\
&\cong \rhom{\scrA}{-[d],\tau^{\leq-d}Q^\bullet}|_{\scrP}\lox_{\scrP}\scrP(P,-)[d-i]\\
&\cong \rhom{\scrA}{P[d],\tau^{\leq-d}Q^\bullet}[d-i]\\
&\cong \rhom{\scrA}{P[i],\tau^{\leq-d}Q^\bullet}
\end{align*}
Thus, the morphism $m$ induces an isomorphism in this case. This completes the proof.
% \textcolor{blue}{Can we state in general case?}
\end{proof}

\begin{lemma}
\label{lem:claim2}
Let $Q^\bullet$ be an object in $\scrK^{[-d,0]}(\scrP)\subset\scrK^{[-d,0]}(\scrA)$. Then we have the following isomorphism in $\der{\scrP}$:
$$\tau^{\leq 0}\rhom{\scrA}{-[d],\tau^{\leq -d}Q^\bullet}|_{\scrP}\cong\tau^{\leq 0} \rhom{\scrA}{-[d],Q^\bullet}|_{\scrP}.$$
\end{lemma}
\begin{proof}
It immediately follows from the following isomorphisms in $\der{\scrP}$:
\begin{align*}
\tau^{\leq 0}\rhom{\scrA}{-[d],\tau^{\leq -d}Q^\bullet}|_{\scrP}
&\cong \tau^{\leq 0} \bigl(\rhom{\scrA}{-,\tau^{\leq -d}Q^\bullet}|_{\scrP}[-d]\bigr)\\
&\cong \tau^{\leq 0} \bigl((\tau^{\leq -d}Q^\bullet)|_{\scrP}[-d]\bigr)\\
&\cong \tau^{\leq 0} \bigl(Q^\bullet|_{\scrP}[-d]\bigr)\\
&\cong \tau^{\leq 0} \bigl(\rhom{\scrA}{-,Q^\bullet}|_{\scrP}[-d]\bigr)\\
&\cong \tau^{\leq 0} \rhom{\scrA}{-[d],Q^\bullet}|_{\scrP}.\qedhere
\end{align*}
\end{proof}

\begin{proposition}
\label{prop:claim3}
Let $P^\bullet, Q^\bullet$ be objects in $\scrK^{[-d,0]}(\scrP)\subset\scrK^{[-d,0]}(\scrA)$. Then we have the following isomorphism in $\der{k}$:
$$\tau^{\leq 0}\rhom{\scrA}{-[d],Q^\bullet}|_{\scrP}\lox_{\scrP}\tau^{\leq 0}\rhom{\scrA}{P^\bullet,-[d]}|_{\scrP\op}\to \tau^{\leq 0}\rhom{\scrA}{P^\bullet, \tau^{\leq -d}Q^\bullet}.$$
\end{proposition}
\begin{proof}
First, we note that $H^{>0}\rhom{\scrA}{P^\bullet, \tau^{\leq -d}Q^\bullet}=0$. Indeed, since $P^\bullet\in \scrK^{[-d,0]}(\scrP)$, it is enough to check that $H^{>0}\rhom{\scrA}{P[i], \tau^{\leq -d}Q^\bullet}=0$ for any $P\in\scrP$ and $0\leq i\leq d$. This is true since $\tau^{\leq -d}Q^\bullet$ has cohomology only in degrees $\leq -d$. Thus, we have natural isomorphisms in $\der{k}$:
$$\tau^{\leq 0}\rhom{\scrA}{P^\bullet, \tau^{\leq -d}Q^\bullet}\cong \rhom{\scrA}{P^\bullet, \tau^{\leq -d}Q^\bullet}.$$
Note that we also have $H^{>0}\rhom{\scrA}{-[d], \tau^{\leq -d}Q^\bullet}|_{\scrP}=0$.
This implies that the morphism 
\begin{align*}
	\tau^{\leq 0}\rhom{\scrA}{-[d],\tau^{\leq -d}Q^\bullet}|_{\scrP}\lox_{\scrP}\tau^{\leq 0}&\rhom{\scrA}{P^\bullet,-[d]}|_{\scrP\op}\\
	&\to \tau^{\leq 0}\bigl(\rhom{\scrA}{-[d],\tau^{\leq -d}Q^\bullet}|_{\scrP}\lox_{\scrP}\rhom{\scrA}{P^\bullet,-[d]}|_{\scrP\op}\bigr)
\end{align*}
is an isomorphism. Hence, by Lemma~\ref{lem:claim1} and Lemma~\ref{lem:claim2}, we have the following isomorphisms in $\der{k}$:
\begin{align*}
\tau^{\leq 0}\rhom{\scrA}{-[d],\tau^{\leq -d}Q^\bullet}|_{\scrP}\lox_{\scrP}\tau^{\leq 0}&\rhom{\scrA}{P^\bullet,-[d]}|_{\scrP\op}\\
&\cong \tau^{\leq 0}\bigl(\rhom{\scrA}{-[d],\tau^{\leq -d}Q^\bullet}|_{\scrP}\lox_{\scrP}\rhom{\scrA}{P^\bullet,-[d]}|_{\scrP\op}\bigr)\\
&\cong \tau^{\leq 0}\bigl(\rhom{\scrA}{-[d],Q^\bullet}|_{\scrP}\lox_{\scrP}\rhom{\scrA}{P^\bullet,-[d]}|_{\scrP\op}\bigr)\\
&\cong \tau^{\leq 0}\rhom{\scrA}{P^\bullet, \tau^{\leq -d}Q^\bullet}.\qedhere
\end{align*}
\end{proof}

\begin{proposition}
\label{prop:standard_triangle}
For any $P^\bullet, Q^\bullet$ in $\scrK^{[-d,0]}(\scrP)\subset\scrK^{[-d,0]}(\scrA)$, we have the following triangle in $\der{k}$:
\begin{align*}
\tau^{\leq 0}\rhom{\scrA}{-[d],Q^\bullet}&|_{\scrP}\lox_{\scrP}\tau^{\leq 0}\rhom{\scrA}{P^\bullet,-[d]}|_{\scrP\op}\\
&\to \tau^{\leq 0}\rhom{\scrA}{P^\bullet,Q^\bullet}
\to \scrA(\Cokd P^\bullet,\Cokd Q^\bullet)\rightdasharrow.
\end{align*}
\end{proposition}
\begin{proof}
First, we already have the following triangle in $\der{k}$:
$$\rhom{\scrA}{P^\bullet, \tau^{\leq -d}Q^\bullet} \to \rhom{\scrA}{P^\bullet, Q^\bullet} \to \rhom{\scrA}{P^\bullet, \tau^{> -d}Q^\bullet} \rightdasharrow.$$
As in the proof of Proposition~\ref{prop:claim3}, $H^{>0}\rhom{\scrA}{-[d],\tau^{\leq -d}Q^\bullet}=0$ holds. Thus, the following is also a triangle in $\der{k}$:
$$\tau^{\leq 0}\rhom{\scrA}{P^\bullet, \tau^{\leq -d}Q^\bullet} \to \tau^{\leq 0}\rhom{\scrA}{P^\bullet, Q^\bullet} \to \tau^{\leq 0}\rhom{\scrA}{P^\bullet, \tau^{> -d}Q^\bullet} \rightdasharrow.$$
The left-hand side of the above triangle is isomorphic to $$\tau^{\leq 0}\rhom{\scrA}{-[d],Q^\bullet}|_{\scrP}\lox_{\scrP}\tau^{\leq 0}\rhom{\scrA}{P^\bullet,-[d]}|_{\scrP\op}$$
by Proposition~\ref{prop:claim3}. On the other hand, we have the following isomorphisms in $\der{k}$ for any $P\in\scrP$ and $0\leq i\leq d$:
\begin{align*}
\tau^{\leq 0}\rhom{\scrA}{P[i], \tau^{> -d}Q^\bullet}
&\cong \tau^{\leq 0}\bigl(\scrA(P,\Cokd Q^\bullet)[-i]\bigr)\\
&\cong \tau^{\leq 0}\rhom{\scrA}{P[i], \Cokd Q^\bullet}
\end{align*}
where the first isomorphism follows from Proposition~\ref{prop:proj_d_cok}. Thus, we have 
$$\tau^{\leq 0}\rhom{\scrA}{P^\bullet, \tau^{> -d}Q^\bullet}\cong \tau^{\leq 0}\rhom{\scrA}{P^\bullet, \Cokd Q^\bullet}$$
in $\der{k}$. By the definition of $d$-iterated cokernel, we also have 
$$\rhom{\scrA}{P^\bullet, \Cokd Q^\bullet}\cong \scrA(\Cokd P^\bullet,\Cokd Q^\bullet)$$
in $\der{k}$. Combining the above isomorphisms, we have the desired triangle in $\der{k}$.
\end{proof}

% \begin{lemma}
% \label{lem:ff}
% The functor $\scrA(-,*)|_{\scrP}\colon \scrA\to \der{\scrP}$ is quasi fully faithful.
% \end{lemma}

\begin{proposition}
\label{prop:standard_triangle_2}
For any $Q^\bullet\in\scrK^{[-d,0]}(\scrP)$, we have the following triangle in $\der{\scrK^{[-d,0]}(\scrP)}$:
\begin{align*}
\scrK^{[-d,0]}(\scrP)(-&,Q^\bullet)|_{\scrP[d]}\lox_{\scrP[d]}\scrK^{[-d,0]}(\scrP)(-, -)|_{(\scrP[d])\op}\\
&\to \scrK^{[-d,0]}(\scrP)(-,Q^\bullet)|_{\scrP}
\to \scrA(\Cokd-,\Cokd Q^\bullet)|_{\scrP}\rightdasharrow.
\end{align*}
\end{proposition}
\begin{proof}
It immediately follows from Proposition~\ref{prop:standard_triangle} and the definition of Hom-complexes in $\scrK^{[-d,0]}(\scrP)$. 
\end{proof}

\begin{proposition}
\label{prop:ess_surj}
The morphism $\Cokd\colon\scrK^{[-d,0]}(\scrP)\to \scrA$ is essentially surjective in $\hqe$.
\end{proposition}
\begin{proof}
Let $A\in \scrA$ be an object. We will show that $A$ is a $d$-iterated cokernel of some object in $\scrK^{[-d,0]}(\scrP)$. Since $\scrA$ has enough projectives, we can take conflations 
$$A^{-(i+1)}\xrightarrow{f^{-(i+1)}} P^{-i}\xrightarrow{g^{-i}} A^{-i}\quad (h^{-i}\colon A^{-(i+1)}\to A^{-i})$$
for $0\leq i\leq d$ where $P^{-i}\in \scrP$ and $A^0:=A$. Then it induces the following diagram in $\per\scrA$. 

$$
\begin{tikzcd}[column sep={12mm,between origins},row sep={13mm}]
A^{-d-1} && A^{-d} && \cdots &&  \cdots && A^{-1} && A^0\\
& P^{-d} && P^{-d+1} && \cdots && \cdots && P^{0} &
\Ar{1-1}{2-2}{"f^{-d-1}"'}
\Ar{2-2}{1-3}{"g^{-d}"}
\Ar{1-3}{2-4}{"f^{-d}"'}
\Ar{2-4}{1-5}{"g^{-d+1}"}
\Ar{2-8}{1-9}{"g^{-1}"}
\Ar{1-9}{2-10}{"f^{0}"'}
\Ar{2-10}{1-11}{"g^0"}
\end{tikzcd}
$$

Thus, by applying the functor $(-)\hut$, we have the following diagram in $\per\scrA\op$:

$$
\begin{tikzcd}[column sep={12mm,between origins},row sep={13mm}]
A^{-d-1}\hut && A^{-d}\hut && \cdots &&  \cdots && A^{-1}\hut && A^0\hut\\
& P^{-d}\hut && P^{-d+1}\hut && \cdots && \cdots && P^{0}\hut &
\Ar{2-2}{1-1}{"f^{-d-1}\hut"'}
\Ar{1-3}{2-2}{"g^{-d}\hut"}
\Ar{2-4}{1-3}{"f^{-d}\hut"'}
\Ar{1-5}{2-4}{"g^{-d+1}\hut"}
\Ar{1-9}{2-8}{"g^{-1}\hut"}
\Ar{2-10}{1-9}{"f^{0}\hut"'}
\Ar{1-11}{2-10}{"g^0\hut"}
\end{tikzcd}
$$
Note that $A^{-i}\hut\in\calA\op$ for any $0\leq i\leq d$ since they are cokernels of morphisms in $\scrA$.

Then, we have the following commutative diagram in $\der{\scrA\op}$:
$$\begin{tikzcd}
 & A^{-d+1}\hut & \\
P^{-d+2}\hut & & P^{-d+1}\hut\\
 & C^{-d+1}\hut &
\Ar{2-1}{1-2}{"{f^{-d+1}\hut}"}
\Ar{1-2}{2-3}{"{g^{-d+1}\hut}"}
\Ar{2-1}{3-2}{"{\varphi^{-d+1}}"'}
\Ar{3-2}{2-3}{"{\psi^{-d+1}}"'}
\Ar{1-2}{3-2}{"{\theta^{-d+1}}"'}
\end{tikzcd}$$
where $C^{-d+1}\hut:=\tau^{\leq 0}\Cocone (g^{-d}\hut\circ f^{-d}\hut)$. The morphism $\psi^{-d+1}\colon C^{-d+1}\hut\to P^{-d+1}\hut$ is a composition of morphisms 
$$C^{-d+1}\hut\to \Cocone (g^{-d}\hut\circ f^{-d}\hut)\to P^{-d+1}\hut.$$
The morphism $\theta^{-d+1}$ is induced by truncating the morphism 
$$\xi^{-d+1}\colon\Cocone(f^{-d}\hut)\to \Cocone(g^{-d}\hut\circ f^{-d}\hut)$$
by $\tau^{\leq 0}$, which makes the following diagram commute in $\der{\scrA\op}$:
$$
\begin{tikzcd}
\Cocone(f^{-d}\hut) & P^{-d+1}\hut & A^{-d}\hut & \Cocone(f^{-d}\hut)[1]\\
\Cocone(g^{-d}\hut\circ f^{-d}\hut) & P^{-d+1}\hut & P^{-d+1}\hut & \Cocone(g^{-d}\hut\circ f^{-d}\hut)[1]
\Ar{1-1}{1-2}{}
\Ar{1-2}{1-3}{"{f^{-d}\hut}"}
\Ar{1-3}{1-4}{}
\Ar{2-1}{2-2}{}
\Ar{2-2}{2-3}{"{g^{-d}\hut\circ f^{-d}\hut}"'}
\Ar{2-3}{2-4}{}
\Ar{1-1}{2-1}{"{\xi^{-d+1}}"}
\Ar{1-2}{2-2}{equal}
\Ar{1-3}{2-3}{"{g^{-d}\hut}"}
\Ar{1-4}{2-4}{"{\xi^{-d+1}[1]}"}
\end{tikzcd}
$$
The morphism $\varphi^{-d+1}$ is a composition of morphisms
$$P^{-d+2}\hut\xrightarrow{f^{-d+1}\hut} A^{-d+1}\hut\xrightarrow{\theta^{-d+1}} C^{-d+1}\hut.$$

In this situation, we can construct the following commutative diagram in $\der{\scrA\op}$ for each $1\leq i\leq d$ inductively:

$$
\begin{tikzcd}
 & A^{-d+i}\hut & \\
P^{-d+(i+1)}\hut & & P^{-d+i}\hut\\
 & C^{-d+i}\hut &
\Ar{2-1}{1-2}{"{f^{-d+i}\hut}"}
\Ar{1-2}{2-3}{"{g^{-d+i}\hut}"}
\Ar{2-1}{3-2}{"{\varphi^{-d+i}}"'}
\Ar{3-2}{2-3}{"{\psi^{-d+i}}"'}
\Ar{1-2}{3-2}{"{\theta^{-d+i}}"'}
\end{tikzcd}
$$
where $C^{-d+i}\hut:=\tau^{\leq 0}\Cocone (\varphi^{-d+(i-1)})$ and $\psi^{-d+i}$ is a composition of morphisms 
$$C^{-d+(i-1)}\hut\to \Cocone (\varphi^{-d+(i-1)})\to P^{-d+i}\hut.$$
The morphism $\theta^{-d+i}$ is induced by truncating the morphism 
$$\xi^{-d+i}\colon\Cocone(f^{-d+(i-1)}\hut)\to \Cocone(\varphi^{-d+(i-1)})$$ 
by $\tau^{\leq 0}$, which makes the following diagram commute in $\der{\scrA\op}$:
$$
\begin{tikzcd}
\Cocone(f^{-d+(i-1)}\hut) & P^{-d+i}\hut & A^{-d+(i-1)}\hut & \Cocone(f^{-d+(i-1)}\hut)[1]\\
\Cocone(\varphi^{-d+(i-1)}\hut) & P^{-d+i}\hut & C^{-d+(i-1)}\hut & \Cocone(\varphi^{-d+(i-1)}\hut)[1]
\Ar{1-1}{1-2}{}
\Ar{1-2}{1-3}{"{f^{-d+(i-1)}\hut}"}
\Ar{1-3}{1-4}{}
\Ar{2-1}{2-2}{}
\Ar{2-2}{2-3}{"{\varphi^{-d+(i-1)}\hut}"'}
\Ar{2-3}{2-4}{}
\Ar{1-1}{2-1}{"{\xi^{-d+i}}"}
\Ar{1-2}{2-2}{equal}
\Ar{1-3}{2-3}{"{\theta^{-d+(i-1)}}"}
\Ar{1-4}{2-4}{"{\xi^{-d+i}[1]}"}
\end{tikzcd}
$$
The morphism $\varphi^{-d+i}$ is a composition of morphisms 
$$P^{-d+(i+1)}\hut\xrightarrow{f^{-d+i}\hut} A^{-d+i}\hut\xrightarrow{\theta^{-d+i}} C^{-d+i}\hut.$$

We will show that the mapping cocone $\Cocone \theta^{-d+i}$ of the morphism $\theta^{-d+i}\colon A^{-d+i}\hut\to C^{-d+i}\hut$ satisfies 
$\Cocone \theta^{-d+i}\in \der{\scrA\op}^{> -d+i+1}$ for each $1\leq i\leq d$. We use induction on $i$. 

First, we show the case $i=1$. Since the morphism $g^{-d}$ is a $d$-epimorphism in $\scrA$, we have $\Cocone (g^{-d}\hut)\in \der{\scrA\op}^{> -d+1}$. Thus, we also have $\Cocone \bigl(\xi^{-d+1}\bigr)[1]\in \der{\scrA\op}^{> -d+1}$. Since $\theta^{-d+1}$ is obtained by truncating $\xi^{-d+1}$, we can see $\Cocone \theta^{-d+1}\in \der{\scrA\op}^{> -d+2}$ by using the octahedral axiom. The induction step is similar to the case $i=1$. 

Therefore, we have $\Cocone \theta^0\in \der{\scrA\op}^{> 1}$. In particular, $\theta^0\colon A^0\hut\to C^0\hut$ is an isomorphism in $\der{\scrA\op}$. This implies that $A\cong C^0$ in $\scrA$. On the other hand, $C^0$ is a $d$-iterated cokernel of some object in $\scrK^{[-d,0]}(\scrP)$ by the construction and proof of Proposition~\ref{prop:iteraded_ker}. Thus, $A$ is a $d$-iterated cokernel of some object in $\scrK^{[-d,0]}(\scrP)$.
% \color{blue}
% Let $A\in \scrA$ be an object. We will show that $A$ is a $d$-iterated cokernel of some object in $\scrK^{[-d,0]}(\scrP)$. Since $\scrA$ is enough projectives, we can take a conflations 
% $$A^{-(i+1)}\to P^{-i}\to A^{-i}\quad (h^{-i}\colon A^{-(i+1)}\to A^{-i})$$
% for $0\leq i\leq d$ where $P^{-i}\in \scrP$ and $A^0:=A$. Then it induces the following diagram in $\per\scrA\op$:
% $$
% \begin{tikzcd}
% \Cocone && P^{-i}\hut & A^{-(i+1)}\hut & \Cocone [1]\\
% & A^{-i}\hut & & 
% \Ar{1-1}{1-3}{}
% \Ar{1-3}{1-4}{}
% \Ar{1-4}{1-5}{}
% \Ar{2-2}{1-1}{"\varphi^{-i}"}
% \Ar{2-2}{1-3}{}
% \end{tikzcd}
% $$
% for $0\leq i\leq d$ where $A^{-i}\hut$ for some $\varphi^{-i}$ induces an isomorphism $A^{-i}\hut\cong \tau^{\leq 0}\Cocone \varphi^{-i}$ for any $0\leq i\leq d$. Thus, we have the following diagram in $\per\scrA\op$:
% $$
% \begin{tikzcd}
% P^{[-d,0]}\hut && P^{[-d,-1]}\hut && \cdots && P^{[-d,-(d-1)]}\hut && P^{-d}\hut &  
% \end{tikzcd}
% $$ 
% \normalcolor
\end{proof}

\begin{cor}
\label{cor:ideal}
The restricted morphism $\Cok_d\colon\scrK^{[-d,0]}(\scrP)\to \scrA$ of $\Cokd\colon\scrK^{[-d,0]}(\scrA)\to \scrA$ defines a DG-quotient of $\scrK^{[-d,0]}(\scrP)$ by $\scrP[d]$. In particular, we have an isomorphism in $\hqe_\bbV$:
$$\scrK^{[-d,0]}(\scrP)\lquot\scrP[d]\cong \scrA.$$
\end{cor}
\begin{proof}
By Proposition~\ref{prop:ess_surj}, the morphism $\Cokd\colon\scrK^{[-d,0]}(\scrP)\to \scrA$ is essentially surjective in $\hqe$. Moreover, it is easy to check that the kernel of $\Cokd\colon\scrK^{[-d,0]}(\scrP)\to \scrA$ is precisely $\scrP[d]$. Thus, by combining Proposition~\ref{prop:standard_triangle_2} and Proposition~\ref{prop:char_quot_functor}, we can see that the morphism $\Cokd\colon\scrK^{[-d,0]}(\scrP)\to \scrA$ defines a DG-quotient of $\scrK^{[-d,0]}(\scrP)$ by $\scrP[d]$. This completes the proof.
\end{proof}

\begin{cor}
\label{cor:extended_module}
Let $\scrC$ be a connective DG-category and $\demdg{\scrC}$ be the $d$-extended module category of $\scrC$. Then we have the following isomorphism in $\hqe_\bbV$:
$$\scrK^{[-d,0]}(\scrP)\lquot\scrP[d]\cong \demdg{\scrC}$$
where $\scrP$ is the full sub DG-category of $\demdg{\scrC}$ consisting of projective objects.
\end{cor}
\begin{proof}
It immediately follows from Corollary~\ref{cor:ideal}. 
\end{proof}

\begin{definition}
\label{def:compact_proj_gen}
Let $\scrA$ be an abelian $d$-truncated DG-category with arbitrary coproducts indexed by a $\bbU$-small set in $H^0(\scrA)$. A full sub DG-category $\scrC$ of $\scrA$ is called a \emph{compact projective generating set} if the following conditions are satisfied:
\begin{enumerate}
  \item $\scrC$ is a $\bbU$-small DG-category.
	\item Any object $C\in\scrC$ is a $\bbU$-compact projective object in $H^0(\scrA)$ that is, the functor 
	$$H^0(\scrA)(C,-)\colon H^0(\scrA)\to \Mod k$$ 
	preserves arbitrary coproducts indexed by a $\bbU$-small set.
	\item For any object $A\in \scrA$, there exists an $d$-epimorphism $\bigoplus_{i\in I} C_i\to A$ in $\scrA$ for some family of objects $\{C_i\}_{i\in I}$ in $\scrC$.
\end{enumerate}
\end{definition}

\begin{theorem}
\label{thm:morita}
Let $\scrA$ be an abelian $d$-truncated DG-category such that $H^0(\scrA)$ has arbitrary coproducts and a compact projective generating set $\scrC\subset\scrA$. Then the morphism $\scrA(-,*)|_{\scrC}\colon \scrA\to \demdg{\scrC}$ is an isomorphism in $\hqe$.
\end{theorem}
\begin{proof}
First, we note that the full sub DG-category $\scrP$ of projective objects in $\scrA$ corresponds to the full sub DG-category consisting of direct summands of infinite direct sums of objects in $\scrC$. Since $\scrC$ is a compact projective generating set in $\scrA$, the full sub DG-category of projective objects in $\demdg{\scrC}$ is quasi-equivalent to $\scrP$ by the DG-functor $\scrA(-,*)|_{\scrC}\colon \scrA\to \demdg{\scrC}$. We identify the full sub DG-category of projective objects in $\demdg{\scrC}$ with $\scrP$ via the above quasi-equivalence.

Thus, by Corollary~\ref{cor:ideal}, we have the following isomorphisms in $\hqe$:
$$\Cokd\colon \scrK^{[-d,0]}(\scrP)\lquot\scrP[d]\xrightarrow{\sim}\scrA \quad\tau^{> -d}\colon \scrK^{[-d,0]}(\scrP)\lquot\scrP[d] \xrightarrow{\sim} \demdg{\scrC}.$$
Thus, we need to check that the following diagram in $\hqe$ is commutative:
$$
\begin{tikzcd}
\scrK^{[-d,0]}(\scrP) \ar{r}{\Cokd} & \scrA \ar{d}{\scrA(-,*)|_{\scrC}}\\
\scrK^{[-d,0]}(\scrP) \ar{r}{\tau^{> -d}} & \demdg{\scrC}
\Ar{1-1}{2-1}{equal}
\end{tikzcd}
$$
To prove the commutativity of the above diagram, it is enough to check that the following are isomorphic in $\der{\scrK^{[-d,0]}(\scrP)\op\lox_k \scrC}$ :
$$\scrA(-,\Cokd*)|_{\scrC}\cong \tau^{> -d}(-)|_{\scrC}$$
This is a direct consequence of Proposition~\ref{prop:proj_d_cok}. This completes the proof.
\end{proof}

\begin{definition}
\label{def:fp}
Let $\scrC$ be a small connective DG-category. An object $M\in \demdg{\scrC}$ is called \emph{finitely presented} if there exists an object $P^\bullet\in \scrK^{[-d,0]}(\scrC)$ such that $M$ is isomorphic to $\tau^{> -d}P^\bullet$.
\end{definition}

\begin{cor}
\label{cor:morita}
Let $\scrA$ be a locally $K$-projective abelian $d$-truncated DG-category with progenerating subcategory $\scrB$ (i.e., $\scrB$ is a small full sub DG-category of $\scrA$ consisting of projective objects such that any object in $\scrA$ admits a $d$-epimorphism from a finite direct sum of objects of $\scrB$). Assume that $H^0(\scrA)$ is essentially small. Then the following DG-functor is a quasi-equivalence:
$$\scrA(-,*)|_{\scrB}\colon \scrA\to \demfpdg{\scrB}$$
where $\demfpdg{\scrB}$ is the finitely presented $d$-extended module category of $\scrB$.
\end{cor}
\begin{proof}
The proof is similar to the proof of Theorem~\ref{thm:morita} by replacing $\scrC$ with $\scrB$ and $\demdg{\scrC}$ with $\demfpdg{\scrB}$. 
\end{proof}

\begin{remark}
\label{rem:morita}
In general, the full sub DG-category $\demfpdg{\scrB}$ of $\demdg{\scrB}$ is not an abelian $d$-truncated DG-category. 
\end{remark}

\begin{cor}
\label{cor:morita_2}
Let $k$ be a field and $\scrA$ be a locally finite abelian $d$-truncated DG-category over $k$ with progenerating subcategory $\scrB$. Assume that $H^0(\scrA)$ is essentially small. Then the following morphism is an isomorphism in $\hqe$:
$$\scrA(-,*)|_{\scrB}\colon \scrA\to \demfddg{\scrB}$$
where $\demfddg{\scrB}$ is the finite-dimensional $d$-extended module category of $\scrB$, that is, the full sub DG-category of $\demdg{\scrB}$ consisting of objects $M$ such that $H^i(M)$ is finite-dimensional for any $i\in\bbZ$.
\end{cor}
\begin{proof}
One can easily check that $M\in\dem{\scrB}$ is finitely presented if and only if it has finite-dimensional cohomologies for a locally finite DG-category $\scrB$. Thus, we have $\demfddg{\scrB}=\demfpdg{\scrB}$. Therefore, the morphism $\scrA(-,*)|_{\scrB}\colon \scrA\to \demfddg{\scrB}$ is an isomorphism in $\hqe$ by Corollary~\ref{cor:morita}.
\end{proof}

\bibliographystyle{mybstwithlabels.bst}
\bibliography{myrefs}

\end{document}